\documentclass[11pt]{amsart}
\usepackage{tikz}
\usepackage{mathrsfs,dsfont}
\usepackage[mathcal]{euscript}
\usepackage{graphicx,graphics,epsfig}
\usepackage{amsmath,amssymb,amsfonts,amsthm,amscd}
\usepackage{calc}
\usepackage{color}
\usepackage{fancyhdr}
\usepackage{anysize}
\usepackage{enumerate}
\usepackage{indentfirst,latexsym}
\usepackage[all,poly,knot]{xy}
\usepackage{lineno}
\usepackage{etoolbox}
\usepackage[pagebackref]{hyperref}

\newcommand{\comment}[1]{}

\numberwithin{equation}{section}
\theoremstyle{plain}
\newtheorem{theorem}{Theorem}[section]
\newtheorem{lemma}[theorem]{Lemma}

\newtheorem{proposition}[theorem]{Proposition}
\newtheorem{corollary}[theorem]{Corollary}
\newtheorem{question}[theorem]{Question}
\newtheorem{problem}[theorem]{Problem}
\newtheorem{conjecture}[theorem]{Conjecture}

\newcommand{\stepref}[2]{\hyperlink{#1}{Step~#2}}

\theoremstyle{definition}
\newtheorem{definition}[theorem]{Definition}

\newtheorem{remark}[theorem]{Remark}

\newtheorem{observation}[theorem]{Observation}

\AtBeginEnvironment{proof}{\setcounter{step}{0}}

\begin{document}
%\linenumbers
\title[Extension of pluricanonical forms]
{Matsumura's extension problem for pluricanonical forms
in K\"ahler families I:
the smooth and essentially Moishezon cases}

\author[Jian Chen]{Jian Chen}
\address{Jian Chen, School of Mathematics and Statistics,
Central China Normal University, Wuhan 430079,
People's Republic of China}
\email{jian-chen@whu.edu.cn}

\author[Sheng Rao]{Sheng Rao}
\address{Sheng Rao, Department of Mathematics, Wuhan University, Hubei, Wuhan 430074, China}
\email{likeanyone@whu.edu.cn}

\author[Kai Wang]{Kai Wang}
\address{Kai Wang, School of Mathematics and statistics, Wuhan University, Wuhan 430072, China}
\email{kaiwang@whu.edu.cn}

\date{\today}

\thanks{The first author is supported by  NSFC  (Grant No. 12601144). The second and third authors are partially supported by NSFC (Grant No. 12271412, W2441003, 12671107) and Hubei Provincial Innovation Research Group Project (Grant No. 2025AFA044).}
\subjclass[2020]{32J18, 32Q15, 32W05, 14D07}
\keywords{Plurigenera, K\"ahler family, nef canonical bundle,
Monge--Amp\`ere equation, extension of sections}

\begin{abstract} 
In this paper, we study a problem posed by Matsumura on the extension of
pluricanonical forms in K\"ahler families with a relatively nef
canonical bundle.
We give an affirmative answer in the smooth case,  and for
one-parameter degenerations under the additional assumption that the
relative canonical bundle restricts to a big line bundle on an
irreducible component of the central fiber.
In particular, we confirm Siu's conjecture on the invariance of
plurigenera for a smooth K\"ahler family with one fiber admitting the nef canonical bundle.
The key step in the smooth case is to construct a semipositively
curved singular metric on the relative canonical bundle with
the integrability required for Cao's $L^2$ extension theorem,
using P\u{a}un's twisted form of Schumacher's curvature formula
and a uniform Monge--Amp\`ere estimate obtained by adapting
the capacity method of Boucksom--Eyssidieux--Guedj--Zeriahi.
For one-parameter degenerations, we propagate bigness from
an irreducible component of the central fiber to nearby smooth
fibers and then apply Matsumura's relative Kawamata--Viehweg
vanishing theorem on a projective modification to obtain the desired
extension.

\end{abstract}
\maketitle

\setcounter{tocdepth}{1}
\tableofcontents

\section{Introduction}

Throughout this paper, we call a proper morphism $f:X\to Y$
of complex manifolds \emph{K\"ahler} if there exists a
smooth real $d$-closed $(1,1)$-form $\Omega$ on $X$ whose
restriction to every fiber of $f$ is a K\"ahler form. After shrinking
$Y$, the source of a K\"ahler morphism is itself K\"ahler.

Motivated by Y.-T. Siu's conjecture on the invariance of plurigenera~\cite{Si02a} and the DLT extension conjecture~\cite[Conjecture 1.3]{DHP13} arising in the minimal model program,
S. Matsumura posed the following problem.

\begin{problem}[{\cite[Problem 1.8]{Mat22}}]\label{prop-matsu}
Let $\pi:X\to\Delta$ be a surjective proper K\"ahler morphism with connected fibers
from a connected complex manifold $X$ to the unit disk $\Delta\subset\mathbb C$,
whose central fiber $X_0:=\pi^{-1}(0)$ is a simple normal crossing
divisor. We call such a morphism a \emph{one-parameter degeneration}.
Assume that $K_X$ is $\pi$-nef. Can every section in $H^0\bigl(X_0,\mathcal O_{X_0}(K_X^m)\bigr)$ be extended to one in $H^0\bigl(X,\mathcal O_X(K_X^m)\bigr)$, after replacing
$\Delta$ by a smaller disk if necessary?
\end{problem}

\subsection{The smooth family case}

In the smooth family case, Matsumura's problem is closely related
to Siu's conjecture on the invariance of plurigenera.
Beginning with the remarkable work of Siu~\cite{Si98,Si02},
substantial progress has been made on the deformation invariance
of plurigenera for locally projective or Moishezon families.
Further invariance and extension results in various settings
were established by M. P\u{a}un~\cite{Pau07},
B. Claudon~\cite{Cla07}, S. Takayama~\cite{Tk07},
B. Berndtsson--P\u{a}un~\cite{BP12},
S. Rao--I.-H. Tsai~\cite{RT22}, B. He--X. Zhou~\cite{HZ23}. In particular, M.-L. Li--Rao--K. Wang--M. Wang \cite[Proposition A.1]{Ra26} prove the deformation invariance of plurigenera for a smooth family with uncountably many Moishezon fibers.

In the K\"ahler setting, Siu~\cite[Conjecture 2.1]{Si02a}
conjectured that, for a smooth family $\pi:X\to \Delta$
with K\"ahler fibers (or with K\"ahler total space), the \emph{$m$-genus} $P_m(X_t):=\dim H^0(X_t,mK_{X_t})$ is independent of
$t\in S$ for each positive integer $m$.
Already in the 1980s, M. Levine~\cite{Lv83,Lv85} obtained
invariance results under additional assumptions for families
whose fibers are compact complex manifolds in Fujiki's
class $\mathscr C$. Recently, J.-P. Demailly~\cite{Dem22} proposed an approach to the
conjecture using infinite-dimensional Bergman bundles, identifying additional
$L^2$ estimates for vertical derivatives as a key missing ingredient.
J. Cao--P\u{a}un~\cite{CP23} established infinitesimal extension
using $L^2$ methods and a generalized
$\partial\bar{\partial}$-lemma, generalizing Levine's results in the K\"ahler setting. The generalized $\partial\bar{\partial}$-lemma part is much connected to \cite{lrw}. Recently, the deformation invariance of plurigenera for a smooth K\"ahler family of threefolds was given in \cite{HLR26}, while the canonical singularities case is obtained in \cite{FLR26}.

Our first main result gives an affirmative
answer to Problem~\ref{prop-matsu} in the smooth case.

\begin{theorem}[{=Theorem~\ref{prop-section-extension}}]\label{thm-main}
Let $f: X\to \Delta$ be a smooth proper surjective K\"ahler
morphism of complex manifolds with connected fibers.
Assume that $K_{X_t}$ is nef for each $t\in \Delta$.
Then, for each integer $m\geq1$, the restriction map
\[
H^0\bigl(X,mK_{X/\Delta}\bigr)
\longrightarrow H^0(X_0,mK_{X_0})
\]
is surjective. Equivalently, for each positive integer $m$, the $m$-genus $P_m(X_t)$ of $X_t$ is constant, by Theorem \ref{intro-main-criterion}.
\end{theorem}

To prove Theorem~\ref{thm-main}, we follow the extension strategy
of Siu--P\u{a}un and use the Ohsawa--Takegoshi extension theorem
established by Cao~\cite[Theorem 1.1]{Cao17} (=
Theorem~\ref{thm-ot}).
The main difficulty lies in ensuring the required integrability
while maintaining semipositive curvature.
To address this, we construct a singular metric on the relative
canonical bundle using P\u{a}un's twisted form of Schumacher's
curvature formula~\cite[Section 3.2, equation (35)]{Pau17} and a
uniform Monge--Amp\`ere estimate obtained
by adapting the capacity method of
S. Boucksom--P. Eyssidieux--V. Guedj--A. Zeriahi~\cite{BEGZ10}.

Motivated by the work of Li--Rao--Wang \cite{LRW25},
Li--X.-L. Liu--Rao \cite{LLR26}, and L. Wang \cite{Wang26}
on the openness of the canonical nef locus, we establish
the local stability of canonical nefness in smooth fiberwise
K\"ahler families. This result is of independent interest
and allows us to weaken the nefness assumption in
Theorem~\ref{thm-main} after shrinking the base.

\begin{theorem}[{=Theorem~\ref{thm-nef-openness}}]\label{intro-thm-nef-openness}
Let $\pi:X\to \Delta$ be a smooth proper surjective
holomorphic map of complex manifolds with connected fibers
$X_t:=\pi^{-1}(t)$.
Suppose that $X_0$ is K\"ahler and that $K_{X_0}$ is nef.
Then there is $\varepsilon>0$ such that $K_{X_t}$ is nef for each
$t\in \Delta_{\varepsilon}:=\{t\in\Delta:|t|<\varepsilon\}$.

\end{theorem}

Theorem~\ref{thm-main}, combined with
Theorem~\ref{intro-thm-nef-openness}, easily yields the following.
\begin{theorem}\label{intro-thm-main}
Let $f:X\to \Delta$ be a smooth proper surjective K\"ahler
morphism with connected fibers.
Assume that $K_{X_0}$ is nef.
Then, for each positive integer $m$, the $m$-genus 
$P_m(X_t)$ of $X_t$ is constant on a neighborhood of $0$ that can be chosen
independently of $m$.
\end{theorem}

Based on all the above argument,
it is natural to propose:
\begin{conjecture}
Let $\pi:X\to\Delta$ be a smooth family of compact K\"ahler manifolds.
Suppose that one fiber has a semipositive canonical line bundle.
Then $K_X$ is semipositive (or even semiample) in a neighborhood of every fiber of $\pi$.
\end{conjecture}

By the local K\"ahler stability theorem of Kodaira--Spencer,
one may wish to weaken the K\"ahler morphism assumption in
Theorem~\ref{intro-thm-main} to the assumption that $X_0$ is K\"ahler.
However, for such a family, the total space may fail to be K\"ahler
even after shrinking the base; see, for example,~\cite[\S~5]{C26}.
In the proof of Theorem~\ref{intro-thm-main}, however, local
K\"ahlerness of the total space is used in both the curvature and
extension arguments. In the $L^2$-extension approach, solving the
$\bar\partial$-equation without the K\"ahler identities on the total
space appears difficult.
A natural question arises:

\begin{question}
Let $f:X\to\Delta$ be a smooth proper surjective holomorphic
map of complex manifolds with connected fibers, and write
$X_t:=f^{-1}(t)$.
Assume that $X_0$ is K\"ahler and that $K_{X_0}$ is nef.
Can one prove, using Kodaira--Spencer deformation theory
(for instance, via a power-series method),
that for each integer $m\geq 1$, the $m$-genus 
$P_m(X_t)$ of $X_t$ 
is constant in a neighborhood of $0$ (possibly depending on $m$)?
\end{question}

\subsection{The one-parameter degeneration case}

We now turn to one-parameter degenerations. Matsumura~\cite[Corollary 1.9]{Mat22} obtained an affirmative
answer to Problem~\ref{prop-matsu} under the additional
assumption that $K_X$ admits a singular Hermitian metric
with semipositive curvature and zero Lelong numbers at
each point of the central fiber.
Matsumura’s proof combines an injectivity theorem with multiplier
ideal sheaves and the restriction exact sequence.

Rather than assuming such a metric, we use relative nefness
to propagate bigness from one irreducible component of the
central fiber to nearby smooth fibers  (thus the morphism is essentially Moishezon near the central fiber).
We then apply Matsumura's relative
Kawamata--Viehweg--Nadel vanishing theorem~\cite[Theorem 1.7]{Mat22}
(=Theorem~\ref{thm-relative-kv-vanishing}) on a suitable proper modification
that is projective over the base, and descend the resulting
vanishing to $X$.
For $m\geq2$, this gives the vanishing of all higher direct
images of $mK_{X/\Delta}$ after shrinking the disk;
extension then follows from the restriction exact sequence.
We obtain the following result.

\begin{theorem}[{=Theorem~\ref{thm-nef-big}}]\label{intro-thm-nef-big}
Let $\pi:X\to \Delta$ be a proper surjective K\"ahler morphism with connected fibers
from a connected complex manifold of dimension $n+1$, where $n\geq1$.
Suppose that the central fiber $X_0:=\pi^{-1}(0)$ is a 
 simple normal crossing divisor and that $K_X$ is
$\pi$-nef. Write $X_0=\sum_{i=1}^aD_i$.
Assume that $K_{X/\Delta}|_{D_i}$ is big for at least one index $i$.
Then the following statements hold.
\begin{enumerate}[(i)]
\item
For each fixed integer $m\geq2$, after shrinking the disk 
centered at $0$, we obtain that  
\begin{equation*}
R^q(\pi)_*K_{X/\Delta}^{\otimes m}=0
\qquad\text{for each }q>0.
\end{equation*}
\item
For each fixed integer $m\geq1$, after shrinking the disk 
centered at $0$, we obtain that  
the restriction map
\begin{equation*}
H^0(X,K_{X/\Delta}^{\otimes m})
\longrightarrow
H^0(X_0,K_{X/\Delta}^{\otimes m}|_{X_0})
\end{equation*}
is surjective. 
\end{enumerate}
\end{theorem}

The rest of the paper is organized as follows.
In Section~\ref{sec-preli}, we recall basic notions from
pluripotential theory and establish a uniform Monge--Amp\`ere
estimate relative to the minimal-singularity envelope.
In Section~\ref{sect-exten-smooth}, we use this estimate to
construct a semipositively curved singular metric on the relative
canonical bundle and verify the integrability required for
Cao's $L^2$ extension theorem, thereby proving
Theorem~\ref{thm-main}. Section~\ref{sec-nef-openness} establishes
the local stability of canonical nefness by combining stable-map compactness,
deformation theory of rational curves, and a hard Lefschetz
dimension estimate, without assuming that the total space
is K\"ahler.
Finally, Section~\ref{sec-nef-big} proves
Theorem~\ref{intro-thm-nef-big}
by propagating bigness of the relative canonical bundle from
one central component to nearby smooth fibers and applying
Matsumura's relative vanishing theorem on a projective modification.

\section{Some preliminary estimates}\label{sec-preli}

Throughout this paper, we use the convention
$dd^c=\sqrt{-1}\,\partial\bar\partial$.
Throughout this section, $(X,\omega)$ is a connected compact K\"ahler manifold of dimension
$n\geq1$, with reference K\"ahler form $\omega$.
Let $dV$ be a fixed smooth positive volume form on $X$.
Let $H_{\partial\bar{\partial}}^{1,1}(X,\mathbb R)$ denote the
real Bott--Chern cohomology group.

\subsection{Basic notions in pluripotential theory}
In this subsection, we recall some basic definitions and notation
from pluripotential theory.

\begin{definition}[{e.g.,~\cite[Definition 1.5]{DP04}}]\label{def-nef}
A class $\alpha\in
H_{\partial\bar{\partial}}^{1,1}(X,\mathbb R)$ is said to be
\emph{nef} if, for a smooth $d$-closed real $(1,1)$-form
$\eta$ representing $\alpha$ and every $\delta>0$, there exists
$\psi_\delta\in C^\infty(X,\mathbb R)$ such that
\[
\eta+dd^c\psi_\delta\geq-\delta\omega.
\]

\end{definition}

Throughout the rest of this section, fix a nef class
$\alpha\in H_{\partial\bar{\partial}}^{1,1}(X,\mathbb R)$
and a smooth $d$-closed real $(1,1)$-form $\theta$
representing $\alpha$.
For each $0<\varepsilon\leq1$, set
\[
\theta_\varepsilon:=\theta+\varepsilon\omega.
\]
The class $\{\theta_\varepsilon\}$ is K\"ahler, although
$\theta_\varepsilon$ itself need not be positive.
We denote its top self-intersection number by
\begin{equation*}
\mathcal V_\varepsilon
:=
\int_X\{\theta_\varepsilon\}^n.
\end{equation*}

\begin{definition}[{e.g.,~\cite[\S~1.3]{BEGZ10}}]
For a smooth $d$-closed real $(1,1)$-form $\eta$, we denote by
$\operatorname{PSH}(X,\eta)$ the space of \emph{$\eta$-plurisubharmonic
($\eta$-psh) functions} on $X$, namely, upper semicontinuous functions
$u:X\to\mathbb R\cup\{-\infty\}$ which are locally the sum of a
plurisubharmonic function and a smooth function, are not identically
$-\infty$, and satisfy
$
\eta+dd^cu\ge0
$
in the sense of currents.
\end{definition}

\begin{definition}[{e.g.,~\cite[\S~1.4]{BEGZ10}}]
For a \emph{pseudoeffective} class $\{\eta\}$, i.e., a class admitting a
closed positive $(1,1)$-current as a representative, define the
\emph{envelope}
\[
P_\eta
:=
\left(
\sup\left\{
u\in\operatorname{PSH}(X,\eta):u\le0
\right\}
\right)^*,
\]
where $^*$ denotes the \emph{upper semicontinuous regularization}, i.e.,
\[
f^*(x):=\limsup_{y\to x}f(y)
=
\inf\left\{
\sup_{y\in U}f(y):
U\subset X\text{ is an open neighborhood of }x
\right\}.
\]
\end{definition}

Clearly, by the definition of the envelope and the stability of
plurisubharmonic functions under upper envelopes, we have
$P_\eta\in\operatorname{PSH}(X,\eta)$ and $\sup_XP_\eta=0$.
Moreover, $P_\eta$ has minimal singularities in the class
$\{\eta\}$ in the sense that, for any $v\in\operatorname{PSH}(X,\eta)$,
\[
v\le P_\eta+\sup_Xv
\]
on $X$.

\begin{definition}[{e.g.,~\cite[Definitions 1.17 and 2.1]{BEGZ10}}]
The \emph{volume of a big class} $\{\eta\}$ is defined by
\[
\operatorname{Vol}(\{\eta\})
:=
\int_X\left\langle
(\eta+dd^cP_\eta)^n
\right\rangle.
\]
A function $u\in\operatorname{PSH}(X,\eta)$ is said to have
\emph{full Monge--Amp\`ere mass} if
\[
\int_X\left\langle
(\eta+dd^cu)^n
\right\rangle
=
\operatorname{Vol}(\{\eta\}).
\]
If $\{\eta\}$ is nef and big, then
\[
\operatorname{Vol}(\{\eta\})=\int_X\{\eta\}^n.
\]
\end{definition}

Here the brackets denote the
\emph{nonpluripolar product} of~\cite[\S~1.2]{BEGZ10} and, in particular,
$\left\langle(\eta+dd^cu)^n\right\rangle$ is a positive measure
which assigns zero measure to pluripolar sets and agrees with the usual
Bedford--Taylor product wherever $u$ is locally bounded.

\begin{definition}[{e.g.,~\cite[\S~1.2, \S~4.1, (4.3)]{BEGZ10}}]
For a big class $\{\eta\}$, we define the \emph{Monge--Amp\`ere capacity}
\begin{equation*}
\operatorname{Cap}_\eta(E)
:=
\sup_{\substack{u\in\operatorname{PSH}(X,\eta)\\
P_\eta-1\le u\le P_\eta}}
\int_E
\left\langle(\eta+dd^cu)^n\right\rangle,
\end{equation*}
where $E\subset X$ is a Borel set.

The \emph{normalized Monge--Amp\`ere capacity} is defined by
\begin{equation*}
\operatorname{cap}_\eta(E):=
\frac{\operatorname{Cap}_\eta(E)}
{\operatorname{Vol}(\{\eta\})}.
\end{equation*}
Since
$\operatorname{Cap}_\eta(X)=\operatorname{Vol}(\{\eta\})$,
we have
\[
0\le\operatorname{cap}_\eta(E)\le1.
\]
\end{definition}

In the definition of the Monge--Amp\`ere capacity, we use
$P_\eta-1\le u\le P_\eta$ for convenience in the sublevel-set estimates below.
This differs slightly from the convention
$P_\eta\le u\le P_\eta+1$ in~\cite[\S~4.1, (4.3)]{BEGZ10}
only by adding a constant to the test functions, and therefore
defines exactly the same capacity.

\subsection{Some preliminary estimates}
In this subsection, we prove some important estimates which will be used in section three.

We use uniform version of Skoda's integrability theorem to establish the following inequality:
\begin{lemma}\label{lem-uniform-skoda}
There exist constants $a>0$ and $C>0$, independent of
$\varepsilon$, such that
\[
\int_X e^{-au}\,dV\le C
\]
for any $u\in\operatorname{PSH}(X,\theta_\varepsilon)$ satisfying the normalization
$\sup_Xu=0$.
\end{lemma}

\begin{proof}
Choose $B>0$ such that
\[
B\omega-\theta-\omega\ge0.
\]
Then, for $0<\varepsilon\le1$,
\[
B\omega-\theta_\varepsilon
=
B\omega-\theta-\varepsilon\omega
\ge0.
\]
If $u\in\operatorname{PSH}(X,\theta_\varepsilon)$, then
\[
B\omega+dd^cu
=
(B\omega-\theta_\varepsilon)
+
(\theta_\varepsilon+dd^cu)
\ge0.
\]
Thus $u/B\in\operatorname{PSH}(X,\omega)$, and
$\sup_X(u/B)=0$.

By the uniform version of Skoda's integrability theorem
(e.g.,~\cite[proof of Proposition 4.3]{BEGZ10}), there exist
constants $a_0>0$ and $C_0>0$, depending only on
$(X,\omega,dV)$, such that
\[
\int_X e^{-a_0v}\,dV\le C_0
\]
for any $v\in\operatorname{PSH}(X,\omega)$ with $\sup_Xv=0$.
Applying this to $v=u/B$ and setting
$a:=a_0/B$ and $C:=C_0$ proves the assertion.
\end{proof}

We reformulate~\cite[Lemma 4.2]{BEGZ10} as follows.

\begin{lemma}\label{lem-relative-at}
Let $\eta$ be a smooth $d$-closed real $(1,1)$-form whose class
$\{\eta\}$ is big, and let $K\subseteq X$ be a compact
nonpluripolar set. Set
\[
U_{K,\eta}
:=
\left(
\sup\{u\in\operatorname{PSH}(X,\eta):u\le0\text{ on }K\}
\right)^*,
\qquad
M_{K,\eta}:=\sup_XU_{K,\eta}.
\]
Then
\begin{equation}\label{eq-relative-at}
\frac{\operatorname{Cap}_\eta(K)}
{\operatorname{Vol}(\{\eta\})}
\ge
\frac1{(M_{K,\eta}+1)^n}.
\end{equation}
\end{lemma}

\begin{proof}
By~\cite[Lemma 4.2]{BEGZ10},
\[
\left(
\frac{\operatorname{Vol}(\{\eta\})}
{\operatorname{Cap}_\eta(K)}
\right)^{1/n}
\le
\max\{1,M_{K,\eta}\}.
\]
Since $M_{K,\eta}\ge0$, we have
$\max\{1,M_{K,\eta}\}\le M_{K,\eta}+1$, and~\eqref{eq-relative-at} follows.
\end{proof}

\begin{proposition}
For a Borel set $E\subseteq X$, put
\[
\operatorname{Vol}_{dV}(E):=\int_E dV.
\]
There are constants $a>0$ and $C>0$, independent of $\varepsilon$,
such that, for any Borel set $E\subseteq X$ with
$\operatorname{cap}_{\theta_\varepsilon}(E)>0$,
\begin{equation}\label{eq-exponential-capacity}
\operatorname{Vol}_{dV}(E)
\leq
C\exp\left(
-a\,\operatorname{cap}_{\theta_\varepsilon}(E)^{-1/n}
\right).
\end{equation}
If $\operatorname{cap}_{\theta_\varepsilon}(E)=0$, then
$\operatorname{Vol}_{dV}(E)=0$.

Consequently, for each $N>0$ there is $C_N>0$, independent of
$\varepsilon$, such that
\begin{equation}\label{eq-power-capacity}
\operatorname{Vol}_{dV}(E)
\leq
C_N\operatorname{cap}_{\theta_\varepsilon}(E)^N
\end{equation}
for any Borel set $E\subseteq X$.
\end{proposition}

\begin{proof}
We first prove~\eqref{eq-exponential-capacity} for a compact
nonpluripolar set $K\subseteq X$. Put
\[
U:=U_{K,\theta_\varepsilon},
\qquad
M:=\sup_XU.
\]
Then $U-M\in\operatorname{PSH}(X,\theta_\varepsilon)$ and
$\sup_X(U-M)=0$. Since $U\leq0$ almost everywhere on $K$,
Lemma~\ref{lem-uniform-skoda} gives
\[
e^{aM}\operatorname{Vol}_{dV}(K)
\leq
\int_K e^{-a(U-M)}\,dV
\leq C.
\]
Thus
\[
\operatorname{Vol}_{dV}(K)\leq Ce^{-aM}.
\]
By Lemma~\ref{lem-relative-at},
\[
M+1
\geq
\operatorname{cap}_{\theta_\varepsilon}(K)^{-1/n}.
\]
After absorbing the factor $e^a$ into $C$, we obtain~\eqref{eq-exponential-capacity} for compact nonpluripolar sets.

If $K$ is compact and pluripolar, then
$\operatorname{cap}_{\theta_\varepsilon}(K)=0$ and
$\operatorname{Vol}_{dV}(K)=0$.
Now let $E\subseteq X$ be Borel. For any compact set
$K\subseteq E$, we have
\[
\operatorname{cap}_{\theta_\varepsilon}(K)
\leq
\operatorname{cap}_{\theta_\varepsilon}(E).
\]
By inner regularity,
\[
\operatorname{Vol}_{dV}(E)
=
\sup_{\substack{K\subseteq E\\ K\ \text{compact}}}
\operatorname{Vol}_{dV}(K).
\]
Taking the supremum over compact subsets $K\subseteq E$ proves~\eqref{eq-exponential-capacity}.

Finally,
$0\leq\operatorname{cap}_{\theta_\varepsilon}(E)\leq1$, and for
each $N>0$ there is a constant $C_N>0$ such that
\[
e^{-ax^{-1/n}}\leq C_Nx^N,
\qquad 0<x\leq1.
\]
This proves~\eqref{eq-power-capacity}.
\end{proof}

The proof follows the capacity estimates and sublevel-set argument
in~\cite[\S\S~4.1--4.2]{BEGZ10}.
We use normalized capacities and measures, and track the constants
to verify that they are independent of $\varepsilon$.

\begin{proposition}\label{prop-uniform-czero}
Fix $p>1$ and $A>0$.
Suppose that, for each $0<\varepsilon\leq1$, the following
conditions hold.
\begin{enumerate}[$(1)$]
\item
$f_\varepsilon\in L^p(dV)$ is nonnegative,
$\int_X f_\varepsilon\,dV=1$, and
$\|f_\varepsilon\|_{L^p(dV)}\leq A$.
\item
$\varphi_\varepsilon\in\operatorname{PSH}(X,\theta_\varepsilon)$
has full Monge--Amp\`ere mass,
$\sup_X\varphi_\varepsilon=0$, and
\begin{equation*}
\left\langle
(\theta_\varepsilon+dd^c\varphi_\varepsilon)^n
\right\rangle
=
\mathcal V_\varepsilon f_\varepsilon\,dV.
\end{equation*}
\end{enumerate}
Then there is a constant $C>0$, depending only on
$(X,\omega,\theta,dV)$, $p$, and $A$, such that
\begin{equation}\label{eq-uniform-czero}
P_{\theta_\varepsilon}-C
\leq
\varphi_\varepsilon
\leq
P_{\theta_\varepsilon}
\end{equation}
on $X$ for each $0<\varepsilon\leq1$.
\end{proposition}

\begin{proof}
Fix $0<\varepsilon\leq1$, and put
$P_\varepsilon:=P_{\theta_\varepsilon}$ and
$d\mu_\varepsilon:=f_\varepsilon\,dV$.
Since $\varphi_\varepsilon\leq0$, the definition of the envelope
gives $\varphi_\varepsilon\leq P_\varepsilon\leq0$.
Also, $P_\varepsilon$ is bounded for each fixed $\varepsilon$:
the class $\{\theta_\varepsilon\}$ is K\"ahler, so there is a
smooth $\theta_\varepsilon$-psh function with supremum zero,
and this function lies below $P_\varepsilon$.
No uniform bound for $P_\varepsilon$ is used below.

\smallskip
\noindent\textbf{Step 1. Uniform preliminary estimates.}
Put $p':=p/(p-1)$.
By H\"older's inequality and~\eqref{eq-power-capacity} with $N=2p'$, for any Borel set
$E\subseteq X$ we have
\begin{equation}\label{eq-measure-capacity}
\mu_\varepsilon(E)
\leq
A\operatorname{Vol}_{dV}(E)^{1/p'}
\leq
A_1\operatorname{cap}_{\theta_\varepsilon}(E)^2,
\end{equation}
where $A_1\geq1$ is independent of $\varepsilon$.

For a fixed $p'>0$ and $a>0$, there is a constant $B$ such that $x^{p'}\leq Be^{ax}$ for $x\ge0$, then Lemma~\ref{lem-uniform-skoda} gives a uniform bound for
$\|\varphi_\varepsilon\|_{L^{p'}(dV)}$.
Since $0\leq P_\varepsilon-\varphi_\varepsilon
\leq-\varphi_\varepsilon$,
H\"older's inequality yields
\begin{equation}\label{eq-initial-energy}
\int_X(P_\varepsilon-\varphi_\varepsilon)\,d\mu_\varepsilon
\leq
\|f_\varepsilon\|_{L^p(dV)}
\|\varphi_\varepsilon\|_{L^{p'}(dV)}
\leq C_0,
\end{equation}
where $C_0\geq1$ is independent of $\varepsilon$.

For $s>0$, put
\[
E_s^\varepsilon:=\{\varphi_\varepsilon<P_\varepsilon-s\}.
\]
Since $P_\varepsilon-\varphi_\varepsilon>s$ on
$E_s^\varepsilon$,~\eqref{eq-initial-energy} gives
$\mu_\varepsilon(E_s^\varepsilon)\leq C_0/s$.

\smallskip
\noindent\textbf{Step 2. Comparison of sublevel sets.}
We claim that, for $s>0$ and $0<t\leq1$,
\begin{equation}\label{eq-capacity-comparison}
t^n\operatorname{cap}_{\theta_\varepsilon}
(E_{s+t}^\varepsilon)
\leq
\mu_\varepsilon(E_s^\varepsilon).
\end{equation}

Take any $u\in\operatorname{PSH}(X,\theta_\varepsilon)$
satisfying $P_\varepsilon-1\leq u\leq P_\varepsilon$,
and put $v:=(1-t)P_\varepsilon+tu-s$.
Then $P_\varepsilon-s-t\leq v\leq P_\varepsilon-s$, so
\[
E_{s+t}^\varepsilon
\subseteq
\{\varphi_\varepsilon<v\}
\subseteq
E_s^\varepsilon.
\]
Moreover, positivity of the mixed Monge--Amp\`ere products gives
\[
\left\langle(\theta_\varepsilon+dd^cv)^n\right\rangle
\geq
t^n\left\langle(\theta_\varepsilon+dd^cu)^n\right\rangle.
\]
Since $\varphi_\varepsilon$ has full Monge--Amp\`ere mass,
the comparison principle
(\cite[Corollary 2.3]{BEGZ10}) yields
\[
\begin{aligned}
t^n\int_{E_{s+t}^\varepsilon}
\left\langle(\theta_\varepsilon+dd^cu)^n\right\rangle
&\leq
\int_{\{\varphi_\varepsilon<v\}}
\left\langle(\theta_\varepsilon+dd^cv)^n\right\rangle\\
&\leq
\int_{\{\varphi_\varepsilon<v\}}
\left\langle
(\theta_\varepsilon+dd^c\varphi_\varepsilon)^n
\right\rangle\\
&\leq
\mathcal V_\varepsilon
\mu_\varepsilon(E_s^\varepsilon).
\end{aligned}
\]
Taking the supremum over such $u$ and dividing by
$\mathcal V_\varepsilon>0$ proves~\eqref{eq-capacity-comparison}.
In particular, taking $t=1$ gives
\[
\operatorname{cap}_{\theta_\varepsilon}
(E_{s+1}^\varepsilon)
\leq
\frac{C_0}{s}.
\]

\smallskip
\noindent\textbf{Step 3. Vanishing of a sublevel capacity.}
Define
$g_\varepsilon(s):=
\operatorname{cap}_{\theta_\varepsilon}
(E_s^\varepsilon)^{1/n}$.
This function is nonincreasing.
Combining~\eqref{eq-measure-capacity} and~\eqref{eq-capacity-comparison}, we obtain
\begin{equation}\label{eq-degiorgi-recursion}
t\,g_\varepsilon(s+t)
\leq
B\,g_\varepsilon(s)^2,
\end{equation}
where $B:=A_1^{1/n}$ is independent of $\varepsilon$.

Choose $s_0:=2+(2B)^nC_0$.
The estimate at the end of Step 2 gives
$g_\varepsilon(s_0)\leq(2B)^{-1}$.
For integers $j\geq0$, set
$s_j:=s_0+2(1-2^{-j})$.
Then $s_{j+1}-s_j=2^{-j}\leq1$.

We prove by induction that
$g_\varepsilon(s_j)\leq2^{-j}/(2B)$.
The case $j=0$ holds by the choice of $s_0$.
If the estimate holds for $j$, then~\eqref{eq-degiorgi-recursion} with $t=2^{-j}$ gives
\[
g_\varepsilon(s_{j+1})
\leq
B\,2^j g_\varepsilon(s_j)^2
\leq
\frac{2^{-(j+1)}}{2B}.
\]
This proves the induction.
Since $s_j<s_0+2$ and $g_\varepsilon$ is nonincreasing,
\[
0\leq g_\varepsilon(s_0+2)
\leq g_\varepsilon(s_j)
\leq\frac{2^{-j}}{2B}.
\]
Letting $j\to+\infty$, we obtain
$\operatorname{cap}_{\theta_\varepsilon}
(E_{s_0+2}^\varepsilon)=0$.

\smallskip
\noindent\textbf{Step 4. The pointwise lower bound.}
By~\eqref{eq-power-capacity},
$\operatorname{Vol}_{dV}(E_{s_0+2}^\varepsilon)=0$, so
$\varphi_\varepsilon\geq P_\varepsilon-s_0-2$ almost everywhere. Choose  arbitrary point $x\in X$ and $U_{x}$ an open neighborhood of $x$ such that $\theta_{\varepsilon}=dd^c\rho_{\varepsilon}$ on $U_x$. Then $\varphi_\varepsilon+\rho_{\varepsilon}$ and $P_\varepsilon+\rho_{\varepsilon}$ are plurisubharmonic functions. By mean value inequality,
\begin{equation}\label{mean value}
P_\varepsilon(x)+\rho_{\varepsilon}(x)-s_0-2\leq\dfrac{1}{\text{vol}(U_x)}\int_{U_{x}}(P_\varepsilon+\rho_{\varepsilon}-s_0-2)\leq\dfrac{1}{\text{vol}(U_x)}\int_{U_{x}}(\varphi_\varepsilon+\rho_{\varepsilon}). \end{equation}
Let $U_{x}\to\{x\}$, then the right hand of (\ref{mean value}) converges to $\varphi_\varepsilon(x)+\rho_{\varepsilon}(x)$. So we get $\varphi_\varepsilon\geq P_\varepsilon-s_0-2$ everywhere.
Together with $\varphi_\varepsilon\leq P_\varepsilon$, this proves~\eqref{eq-uniform-czero} with $C:=s_0+2$, independent of $\varepsilon$.
\end{proof}

\begin{corollary}\label{cor-section-adapted}
Assume the conditions and notation of
Proposition~\ref{prop-uniform-czero}.
Let $L$ be a nef holomorphic line bundle on $X$ with
$2\pi c_1(L)=\{\theta\}$, and choose a smooth Hermitian metric $h$
on $L$ with $\sqrt{-1}\,\Theta_h(L)=\theta$.
For an integer $m\geq 1$ and a nonzero section
$s\in H^0(X,L^{\otimes m})$, put
\[
u_s:=\frac1m\log|s|^2_{h^{\otimes m}},
\]
where the norm is induced by $h$.
Normalize $u_s$ by an additive constant so that $u_s\leq0$.
Then, with the same constant $C$ as in
Proposition~\ref{prop-uniform-czero}, we have
\begin{equation*}
\varphi_\varepsilon\geq u_s-C
\end{equation*}
on $X$ for each $0<\varepsilon\leq1$.
\end{corollary}

\begin{proof}
The Poincar\'e--Lelong formula gives
\[
\theta+dd^cu_s
=
\frac{2\pi}{m}[\operatorname{div}(s)]
\geq0.
\]
Thus $u_s\in\operatorname{PSH}(X,\theta_\varepsilon)$ and
$u_s\leq0$, so $u_s\leq P_{\theta_\varepsilon}$.
Proposition~\ref{prop-uniform-czero} therefore gives
$\varphi_\varepsilon\geq P_{\theta_\varepsilon}-C\geq u_s-C$.
\end{proof}

\section{Extension in smooth K\"ahler families}
\label{sect-exten-smooth}

In this section, we prove Theorem~\ref{thm-main}
(=Theorem~\ref{prop-section-extension}), which gives an affirmative
answer to Problem~\ref{prop-matsu} in the smooth case. The basic setup
is as follows.
Let
\[
f:X\to\Delta
\]
be a smooth proper surjective K\"ahler morphism of complex
manifolds with connected fibers $X_t:=f^{-1}(t)$, and assume
that $K_{X_t}$ is nef for each $t\in\Delta$.
Put $n:=\dim X_t\geq1$.

The constructions below are local near $X_0$.
When shrinking the disk, we restrict $f$ accordingly and retain
the same notation; the passage to the original disk is made
in the proof of Theorem~\ref{prop-section-extension}.
Choose the fiberwise positive smooth real $d$-closed
$(1,1)$-form $\Omega$ supplied by the K\"ahler morphism assumption.
By properness, after shrinking $\Delta$ and replacing $\Omega$
by
\[
\Omega+C f^*(\sqrt{-1}\,dt\wedge d\bar t)
\]
for a sufficiently large constant $C>0$, we may assume that
$\Omega$ is a K\"ahler form on $X$.

Using the coordinate $t$ on $\Delta$, we identify
$K_{X/\Delta}$ with $K_X$.
Fix a smooth reference Hermitian metric $h_{X}$ on
$K_{X/\Delta}$, and put
\[
\theta:=\sqrt{-1}\,\Theta_{h_X}(K_{X/\Delta}),
\qquad
\theta_t:=\theta|_{X_t},
\qquad
\Omega_t:=\Omega|_{X_t}.
\]
Let $dV_t$ be the smooth positive volume form on $X_t$
induced by $h_{X}|_{X_t}$; these forms constitute a smooth
relative volume form $dV$.
In relative holomorphic coordinates $(z_1,\ldots,z_n,t)$,
put $\eta:=dz_1\wedge\cdots\wedge dz_n$ and write
\[
|\eta|_{h_X}^2=e^{-\Phi_X}.
\]
Then, with $\eta_t:=\eta|_{X_t}$,
\[
dV_t=e^{\Phi_X|_{X_t}}\,\sqrt{-1}^{n^2}\eta_t\wedge\overline{\eta_t}.
\]
Thus, in these coordinates,
\[
\theta=dd^c\Phi_X,
\qquad
\theta_t=dd^c_{X_t}(\Phi_X|_{X_t}).
\]
Moreover, we have the nef class
\[
\{\theta_t\}=2\pi c_1(K_{X_t}).
\]

\subsection{Fiberwise twisted K\"ahler--Einstein metrics}

We first construct fiberwise twisted
K\"ahler--Einstein metrics.

\begin{lemma}
For each $\varepsilon>0$ and $t\in\Delta$, there is a
unique K\"ahler metric $\omega_{t,\varepsilon}$ on $X_t$
in the class $\{\theta_t\}+\varepsilon\{\Omega_t\}$
satisfying
\begin{equation}\label{eq-twisted-ke}
\operatorname{Ric}(\omega_{t,\varepsilon})
=
-\omega_{t,\varepsilon}+\varepsilon\Omega_t.
\end{equation}
For each fixed $\varepsilon>0$, these metrics depend
smoothly on $(x,t)\in X$.
\end{lemma}

\begin{proof}
Fix $\varepsilon>0$.
For each $t\in\Delta$, nefness of $K_{X_t}$ implies that
$\{\theta_t\}+\varepsilon\{\Omega_t\}$ is a K\"ahler
class, although the particular representative
$\theta_t+\varepsilon\Omega_t$ need not be positive.
Choose $f_{t,\varepsilon}\in C^\infty(X_t,\mathbb R)$
such that
\[
\chi_{t,\varepsilon}
:=
\theta_t+\varepsilon\Omega_t+dd^cf_{t,\varepsilon}
>0,
\]
and put
\[
F_{t,\varepsilon}
:=
f_{t,\varepsilon}
+\log\frac{dV_t}{\chi_{t,\varepsilon}^n}.
\]
Write $\omega_{t,\varepsilon}
=\chi_{t,\varepsilon}+dd^cv$ and fix the additive
constant in $v$. Then equation~\eqref{eq-twisted-ke}
is equivalent to
\[
(\chi_{t,\varepsilon}+dd^cv)^n
=
e^{v+F_{t,\varepsilon}}\chi_{t,\varepsilon}^n,
\qquad
\chi_{t,\varepsilon}+dd^cv>0.
\]
By~\cite[Theorem 4]{Yau78}, this equation admits
a smooth solution $v_{t,\varepsilon}$, which is unique
by the maximum principle.
Thus
$\omega_{t,\varepsilon}
:=\chi_{t,\varepsilon}+dd^cv_{t,\varepsilon}$
is the unique K\"ahler metric in the prescribed class
satisfying~\eqref{eq-twisted-ke}.

To prove smooth dependence, fix $t_0\in\Delta$ and
choose a smooth local trivialization of the family near
$t_0$.
Extend $f_{t_0,\varepsilon}$ smoothly to nearby fibers,
and use its restrictions as $f_{t,\varepsilon}$.
After shrinking the neighborhood of $t_0$, the forms
$\chi_{t,\varepsilon}$ remain positive.
Both $\chi_{t,\varepsilon}$ and $F_{t,\varepsilon}$
then depend smoothly on $t$.

Under this trivialization, consider the nonlinear operator
\[
\mathcal F_{t,\varepsilon}(v)
:=
\log
\frac{(\chi_{t,\varepsilon}+dd^cv)^n}
{\chi_{t,\varepsilon}^n}
-v-F_{t,\varepsilon},
\]
defined on the open set of real-valued potentials
satisfying $\chi_{t,\varepsilon}+dd^cv>0$.
Its derivative with respect to $v$ at the solution
$v_{t_0,\varepsilon}$ is
\[
D_v\mathcal F_{t_0,\varepsilon}
\big|_{v_{t_0,\varepsilon}}(w)
=
(\Delta_{\omega_{t_0,\varepsilon}}-1)w,
\]
where
$\Delta_{\omega}w:=\operatorname{tr}_{\omega}(dd^cw)$.
The maximum principle shows that this operator has
trivial kernel.
Since it is self-adjoint and elliptic on the compact
fiber $X_{t_0}$, elliptic Fredholm theory implies that
\[
\Delta_{\omega_{t_0,\varepsilon}}-1:
C^{k+2,\gamma}(X_{t_0},\mathbb R)
\longrightarrow
C^{k,\gamma}(X_{t_0},\mathbb R)
\]
is an isomorphism for each integer $k\geq0$ and
$0<\gamma<1$.
The Banach space implicit function theorem, together with elliptic
regularity, therefore yields a smooth local family of
solutions near $t_0$.
Fiberwise uniqueness implies that the resulting local
families of metrics agree on overlaps and thus define
a smooth family over $\Delta$.
\end{proof}

Following the notation of Section~\ref{sec-preli},
applied to the fiber $X_t$, we put
\begin{equation*}
\mathcal V_\varepsilon
:=
\int_{X_t}(\theta_t+\varepsilon\Omega_t)^n
=
\int_{X_t}\omega_{t,\varepsilon}^n
>0.
\end{equation*}
Since $\theta$ and $\Omega$ are $d$-closed on
$X$, fiber integration shows that
$\mathcal V_\varepsilon$ is independent of $t$.

Define the smooth potential
$\phi_{t,\varepsilon}
:=\log(\omega_{t,\varepsilon}^n/dV_t)$.
By~\eqref{eq-twisted-ke}, we have
\begin{equation*}
\begin{aligned}
\omega_{t,\varepsilon}
&=
\theta_t+\varepsilon\Omega_t+dd^c\phi_{t,\varepsilon},
\\
\omega_{t,\varepsilon}^n
&=
e^{\phi_{t,\varepsilon}}dV_t.
\end{aligned}
\end{equation*}
In particular, the second identity fixes the additive
constant in the potential.
Define the normalized potential
\begin{equation*}
\psi_{t,\varepsilon}
:=
\phi_{t,\varepsilon}-\log\mathcal V_\varepsilon.
\end{equation*}
Then
\[
\int_{X_t}e^{\psi_{t,\varepsilon}}dV_t=1.
\]

\begin{lemma}\label{lem-density-barrier}
After shrinking $\Delta$ around $0$, there is a constant
$C>0$ such that the following statements hold.
\begin{enumerate}[(i)]
\item\label{item-density-uniform-bounds}
For each $t\in\Delta$ and $0<\varepsilon\leq1$, we have
$\sup_{X_t}\psi_{t,\varepsilon}\leq C$ and
$\int_{X_t}|\psi_{t,\varepsilon}|\,dV_t\leq C$.

\item\label{item-density-section-barrier}
For each integer $m\geq1$ and each nonzero section
$s\in H^0(X_0,mK_{X_0})$, put
\[
u_s
:=
\frac1m\log|s|_{{h_X|_{X_0}^{\otimes m}}}^2
-
\sup_{X_0}
\left(
\frac1m\log|s|_{{{h_X|_{X_0}}^{\otimes m}}}^2
\right).
\]
Then, for each $0<\varepsilon\leq1$,
\begin{equation*}
\psi_{0,\varepsilon}\geq u_s-C.
\end{equation*}
In particular, the constant $C$ is independent of
$\varepsilon$, $m$, and $s$.
\end{enumerate}
\end{lemma}

\begin{proof}
After shrinking $\Delta$ around $0$, choose a constant
$B>0$ such that
$\theta_t+\varepsilon\Omega_t\leq B\Omega_t$
for each $t\in\Delta$ and $0<\varepsilon\leq1$.
On this smaller disk, the measures $dV_t$ and
$\Omega_t^n$ are uniformly comparable.
Moreover, for $b_t:=\int_{X_t}dV_t$, there are
constants $b_-,b_+>0$ such that
$b_-\leq b_t\leq b_+$ for each $t\in\Delta$.

Put $a_{t,\varepsilon}:=\sup_{X_t}\psi_{t,\varepsilon}$
and
$\widehat\psi_{t,\varepsilon}
:=\psi_{t,\varepsilon}-a_{t,\varepsilon}$.
Since
\[
\Omega_t
+dd^c\left(\frac{\widehat\psi_{t,\varepsilon}}B\right)
=
\frac1B
\left(
B\Omega_t-\theta_t-\varepsilon\Omega_t
+\omega_{t,\varepsilon}
\right)
\geq0,
\]
the function $\widehat\psi_{t,\varepsilon}/B$ is
$\Omega_t$-plurisubharmonic and has supremum zero.
The uniform sup--integral comparison for smooth
K\"ahler families
(\cite[Proposition 3.3 and Section 3.2]{DNGG23}),
together with the uniform comparability of $dV_t$ and
$\Omega_t^n$, gives
\[
\int_{X_t}|\widehat\psi_{t,\varepsilon}|\,dV_t
\leq C_1,
\]
where $C_1>0$ is independent of $t$ and $\varepsilon$.

Set
$I_{t,\varepsilon}
:=\int_{X_t}e^{\widehat\psi_{t,\varepsilon}}\,dV_t$.
Applying Jensen's inequality to the probability measure
$dV_t/b_t$, we obtain
\[
\begin{aligned}
I_{t,\varepsilon}
&\geq
b_t\exp\left(
b_t^{-1}
\int_{X_t}\widehat\psi_{t,\varepsilon}\,dV_t
\right)\\
&\geq b_-e^{-C_1/b_-}
=:c>0.
\end{aligned}
\]
Also, $I_{t,\varepsilon}\leq b_t\leq b_+$, since
$\widehat\psi_{t,\varepsilon}\leq0$.
The normalization
$\int_{X_t}e^{\psi_{t,\varepsilon}}\,dV_t=1$
therefore gives
\[
a_{t,\varepsilon}=-\log I_{t,\varepsilon},
\qquad
-\log b_+\leq a_{t,\varepsilon}\leq-\log c.
\]
Thus $|a_{t,\varepsilon}|\leq C_2$ for a uniform
constant $C_2>0$, and
\[
\int_{X_t}|\psi_{t,\varepsilon}|\,dV_t
\leq
\int_{X_t}|\widehat\psi_{t,\varepsilon}|\,dV_t
+b_t|a_{t,\varepsilon}|
\leq C_1+b_+C_2.
\]
This proves the first assertion with a uniform constant.

For the second assertion,
\[
e^{\psi_{t,\varepsilon}}
=
\frac{e^{\widehat\psi_{t,\varepsilon}}}
{I_{t,\varepsilon}}
\leq c^{-1}.
\]
On $X_0$, put $f_\varepsilon:=e^{\psi_{0,\varepsilon}}$.
Then $\int_{X_0}f_\varepsilon\,dV_0=1$, and
\[
\|f_\varepsilon\|_{L^2(dV_0)}^2
\leq
\|f_\varepsilon\|_{L^\infty(X_0)}
\int_{X_0}f_\varepsilon\,dV_0
\leq c^{-1}.
\]
The function $\widehat\psi_{0,\varepsilon}$ is smooth,
belongs to
$\operatorname{PSH}(X_0,\theta_0+\varepsilon\Omega_0)$,
has supremum zero, and satisfies
\[
\left(
\theta_0+\varepsilon\Omega_0
+dd^c\widehat\psi_{0,\varepsilon}
\right)^n
=
\omega_{0,\varepsilon}^n
=
\mathcal V_\varepsilon f_\varepsilon\,dV_0.
\]
Its total Monge--Amp\`ere mass is
$\mathcal V_\varepsilon$, the volume of the K\"ahler
class $\{\theta_0\}+\varepsilon\{\Omega_0\}$.
Thus it has full Monge--Amp\`ere mass, and
Proposition~\ref{prop-uniform-czero} applies with $p=2$.

For each integer $m\geq1$ and each nonzero section
$s\in H^0(X_0,mK_{X_0})$, the above normalization gives
$\sup_{X_0}u_s=0$.
Applying Corollary~\ref{cor-section-adapted} with
$L=K_{X_0}$ and $h=h_X|_{X_0}$, we obtain
\[
\widehat\psi_{0,\varepsilon}\geq u_s-C_3,
\]
where $C_3>0$ is independent of $\varepsilon$, $m$,
and $s$.
Since $a_{0,\varepsilon}\geq-C_2$, it follows that
\[
\psi_{0,\varepsilon}
=
\widehat\psi_{0,\varepsilon}+a_{0,\varepsilon}
\geq u_s-C_3-C_2.
\]
Taking $C>0$ sufficiently large so that it dominates
the constants appearing above proves both assertions.
\end{proof}

We next use P\u{a}un's twisted form of Schumacher's curvature
formula (\cite[Section 3.2, equation (35)]{Pau17}), whereas  
the untwisted argument is due to Schumacher (\cite{Sch12}).
Let $h_\varepsilon$ be the smooth Hermitian metric on
$K_{X/\Delta}$ induced by the normalized relative
volume form $\omega_{t,\varepsilon}^n/\mathcal V_\varepsilon$,
and put
\[
\rho_\varepsilon
:=\sqrt{-1}\,\Theta_{h_\varepsilon}(K_{X/\Delta}).
\]
Set
\[
\tau_\varepsilon:=\rho_\varepsilon+\varepsilon\Omega.
\]
By~\eqref{eq-twisted-ke}, we have
$\tau_\varepsilon|_{X_t}=\omega_{t,\varepsilon}$.

\begin{proposition}\label{prop-schumacher}
The form $\tau_\varepsilon$ is semipositive on $X$.
% Equivalently,
% \begin{equation*}
% \rho_\varepsilon\geq-\varepsilon\Omega.
% \end{equation*}
\end{proposition}

\begin{proof}
In relative holomorphic coordinates $(z_1,\ldots,z_n,t)$, write
\[
\omega_{t,\varepsilon}
=\sqrt{-1}\,\sum_{\alpha,\beta=1}^n
g_{\alpha\bar\beta}\,dz_\alpha\wedge d\bar z_\beta.
\]
Since $\mathcal V_\varepsilon$ is independent of $t$, the
normalization does not change curvature. The definition of
$h_\varepsilon$ therefore gives
\[
dd^c\log\det(g_{\alpha\bar\beta})
=\rho_\varepsilon
=\tau_\varepsilon-\varepsilon\Omega,
\]
where $dd^c$ acts on the total-space variables.

Let $v_\varepsilon$ be the $\tau_\varepsilon$-horizontal lift
of $\partial/\partial t$. This lift is well defined because
the restriction of $\tau_\varepsilon$ to the vertical tangent
bundle is positive definite.
For a real $(1,1)$-form
$\gamma=\sqrt{-1}\,\sum_{a,b}\gamma_{a\bar b}\,dw_a\wedge d\bar w_b$,
we use the Hermitian convention
$\gamma(v,\bar v):=\sum_{a,b}\gamma_{a\bar b}
v^a\overline{v^b}$.
Define the geodesic curvature by
$c_\varepsilon:=\tau_\varepsilon
(v_\varepsilon,\overline{v_\varepsilon})$.

Let $(g^{\bar\beta\alpha})$ denote the inverse matrix, and use
the convention
\[
\Delta_{\omega_{t,\varepsilon}}
:=\sum_{\alpha,\beta=1}^n
g^{\bar\beta\alpha}
\frac{\partial^2}{\partial z_\alpha\partial\bar z_\beta}.
\]
Applying~\cite[Section 3.2, equation (35)]{Pau17} with
$\beta=\varepsilon\Omega$ gives
\begin{equation}\label{eq-schumacher-formula}
(-\Delta_{\omega_{t,\varepsilon}}+1)c_\varepsilon
=
\left|
\bar\partial_{X_t}(v_\varepsilon|_{X_t})
\right|_{\omega_{t,\varepsilon}}^2
+\varepsilon\Omega(v_\varepsilon,\overline{v_\varepsilon}).
\end{equation}
This identity is local on the base, so it applies to the
present proper family over a disk.
Here $\bar\partial_{X_t}$ denotes differentiation along the
fiber, and
$\bar\partial_{X_t}(v_\varepsilon|_{X_t})$ is a
$T^{1,0}X_t$-valued $(0,1)$-form on $X_t$.

The right-hand side of~\eqref{eq-schumacher-formula} is
nonnegative because $\Omega$ is a K\"ahler form on the total
space. The maximum principle on the compact fiber $X_t$
therefore gives $c_\varepsilon\geq0$.
Since the vertical block of $\tau_\varepsilon$ is the positive
form $\omega_{t,\varepsilon}$, the Schur complement criterion gives
\[
\tau_\varepsilon\geq0
\quad\Longleftrightarrow\quad
c_\varepsilon\geq 0.
\]
\end{proof}

\subsection{A singular limit metric with semipositive curvature}

Fix a smaller disk $\Delta'\Subset\Delta$ centered at $0$.
In a local holomorphic frame $\eta$ of $K_{X/\Delta}$,
write $|\eta|_{h_\varepsilon}^2=e^{-\Phi_\varepsilon}$.
Then
\[
\Phi_\varepsilon=\Phi_X+\psi_{t,\varepsilon},
\qquad
dd^c\Phi_\varepsilon=\rho_\varepsilon.
\]
In particular, $\Phi_\varepsilon-\Phi_X$ is a globally defined
smooth function on the total space.

We now construct a singular Hermitian metric with semipositive
curvature.

\begin{proposition}\label{prop-total-compactness}
There exist a sequence $\varepsilon_j\to0$ and a singular
Hermitian metric $h=e^{-\Phi}$ on
$K_{X/\Delta}|_{f^{-1}(\Delta')}$ such that
\begin{enumerate}[(i)]
\item
The local weights satisfy
$\Phi_{\varepsilon_j}\to\Phi$ in $L^1_{\rm loc}$.

\item
$\sqrt{-1}\,\Theta_h(K_{X/\Delta})\ge0$.

\item
For each integer $m\ge1$ and each nonzero section
$s\in H^0(X_0,mK_{X_0})$, there is a constant $C_s$ such that,
writing $s=g\eta^{\otimes m}$ in a local frame $\eta$ and
denoting by $\Phi$ the weight of $h$ in that frame, one has
\begin{equation}\label{eq-limit-barrier}
\Phi|_{X_0}\ge\frac1m\log|g|^2-C_s.
\end{equation}
\end{enumerate}
The same metric $h$ satisfies these inequalities simultaneously
for all $m$ and $s$.
\end{proposition}

\begin{proof}
Lemma~\ref{lem-density-barrier}.\textup{(\ref{item-density-uniform-bounds})}, together with
Fubini's theorem and the smoothness of $\Phi_0$, gives uniform
local upper and $L^1$ bounds for the weights $\Phi_\varepsilon$.
By Proposition~\ref{prop-schumacher},
\[
dd^c\Phi_\varepsilon\ge-\varepsilon\Omega.
\]
On a sufficiently small coordinate chart, choose a smooth
function $r$ with $dd^cr=\Omega$. Then
$\Phi_\varepsilon+\varepsilon r$ is plurisubharmonic.

The local compactness theorem for plurisubharmonic functions, followed by a diagonal
extraction on a countable coordinate cover, gives a sequence
$\varepsilon_j\to0$ for which the local weights converge in
$L^1_{\rm loc}$. The uniform local $L^1$ bounds exclude
convergence to $-\infty$.
Let $\Phi$ denote the plurisubharmonic representative of
each local limit.
The fixed transition identities pass to the limits, so these
weights define a singular Hermitian metric $h$.
Moreover, passing to the limit in the sense of currents gives
\[
\sqrt{-1}\,\Theta_h(K_{X/\Delta})
=dd^c\Phi\ge0.
\]

Locally, the chosen representative satisfies
\[
\Phi=
\left(
\limsup_{j\to\infty}
(\Phi_{\varepsilon_j}+\varepsilon_jr)
\right)^*,
\]
where the upper-semicontinuous regularization is taken on
the total-space coordinate chart. Consequently,
\[
\Phi(x)\ge
\limsup_{j\to\infty}\Phi_{\varepsilon_j}(x)
\]
at every point $x$ of the chart.

Fix an integer $m\ge1$ and a nonzero section
$s\in H^0(X_0,mK_{X_0})$, and put
\[
A_s:=
\sup_{X_0}
\left(
\frac1m\log|s|_{h_X^{\otimes m}}^2
\right).
\]
This number is finite. In a local frame $\eta$, write
$s=g\eta^{\otimes m}$. Lemma~\ref{lem-density-barrier}.\textup{(\ref{item-density-section-barrier})}
gives
\[
\Phi_\varepsilon|_{X_0}
=\Phi_X|_{X_0}+\psi_{0,\varepsilon}
\ge
\frac1m\log|g|^2-A_s-C.
\]
Thus, for every $x\in X_0$ in this chart,
\[
\Phi(x)\ge
\limsup_{j\to\infty}\Phi_{\varepsilon_j}(x)
\ge
\frac1m\log|g(x)|^2-A_s-C.
\]
Taking $C_s:=C+\max\{A_s,0\}$ proves~\eqref{eq-limit-barrier}.
Since the subsequence was chosen independently of $m$ and $s$,
all these inequalities hold for the same metric $h$.
\end{proof}

\subsection{Extension of pluricanonical sections}

We first establish the integrability of an arbitrary pluricanonical
section and then apply an extension theorem.

\begin{lemma}\label{lem-fiber-integrability}
Let $s\in H^0(X_0,mK_{X_0})$, and let $h$ be the metric in
Proposition~\ref{prop-total-compactness}. Then
\begin{equation}\label{eq-fiber-integrability}
\int_{X_0}|s|^2_{h^{m-1}}<\infty.
\end{equation}
\end{lemma}

\begin{proof}
Assume that $s\neq0$. We write locally $s=g\eta^m$, where $\eta$ is a
frame of $K_{X/\Delta}$. From~\eqref{eq-limit-barrier},
\[
|g|^2e^{-(m-1)\Phi}
\le C|g|^{2/m}.
\]
The right-hand side is locally integrable, with no restriction on the multiplicities of
$\operatorname{div}(s)$. This proves~\eqref{eq-fiber-integrability}.
\end{proof}

We now apply the extension theorem to arbitrary pluricanonical sections.
The following result is the qualitative form of Cao's extension theorem
(\cite[Theorem 1.1]{Cao17}).

\begin{theorem}[{\cite[Theorem 1.1]{Cao17}}]\label{thm-ot}
Let $q:Y\to\Delta$ be a proper surjective holomorphic map from a
K\"ahler manifold, where $\Delta\subset\mathbb C$ is a disk centered
at $0$ with coordinate $t$.
Assume that $q$ is a submersion along $Y_0:=q^{-1}(0)$.
Let $L$ be a holomorphic line bundle on $Y$ endowed with a singular
Hermitian metric $h_L$ such that
$\sqrt{-1}\,\Theta_{h_L}(L)\geq0$ in the sense of currents.
Assume that $h_L|_{Y_0}$ is well defined.
Then every section
$u\in H^0(Y_0,K_{Y_0}\otimes L|_{Y_0})$
with finite $L^2$ norm with respect to $h_L|_{Y_0}$
admits an extension $U\in H^0(Y,K_Y\otimes L)$ satisfying
\[
U|_{Y_0}=u\wedge q^*(dt).
\]
\end{theorem}
For a relatively flat coherent sheaf, it characterizes the local constancy of fiberwise cohomology dimensions in terms of cohomological flatness, base change and  
	extension in consecutive degrees, and torsion freeness of the corresponding
	higher direct image sheaves. 
\begin{theorem}[part of {\cite[Theorem 1.2]{FLR26}}]\label{intro-main-criterion}
	Let $f:X\to Y$ be a proper morphism of complex spaces to a smooth Stein curve and $\mathscr F$ a coherent analytic sheaf on $X$ which is flat with respect to $f$. Fix $y\in Y$ and $q\geq0$. Then the following conditions are equivalent:
	\begin{itemize}
			\item[(i)] $\mathscr F$ is cohomologically flat in dimension $q$ over
			$Y$.
			\item[(ii)] The base change maps in degrees $q$ and $q-1$,
			\[
			\varphi^j(y):(R^jf_*\mathscr F)_y
			\otimes_{\mathcal O_{Y,y}}\mathbb C(y)
			\longrightarrow H^j(X_y,\mathscr F_{(y)}),
			\qquad j=q,q-1,
			\]
			are isomorphisms, where $\mathscr F_{(y)}$ denotes the analytic restriction of $\mathscr F$ to the analytic fiber $X_y$.
			\item[(iii)] The function
			\[
			y\longmapsto\dim H^q(X_y,\mathscr F_{(y)})
			\]
			is locally constant.
			\item[(iv)] The restriction maps
			\[
			r^j:H^j(X,\mathscr F)\longrightarrow
			H^j(X_y,\mathscr F_{(y)}),
			\qquad j=q,q-1,
			\]
			are surjective.
			\item[(v)] The sheaves $R^qf_*\mathscr F$ and
			$R^{q+1}f_*\mathscr F$ are torsion free.
	\end{itemize}
\end{theorem}

We are now ready to give an affirmative answer to
Problem~\ref{prop-matsu} in the smooth case.

\begin{theorem}\label{prop-section-extension}
Let $f: X\to \Delta$ be a smooth proper surjective K\"ahler
morphism of complex manifolds with connected fibers.
Assume that $K_{X_t}$ is nef for each $t\in \Delta$.
Then, for each integer $m\geq1$, the restriction map
\[
H^0\bigl(X,mK_{X/\Delta}\bigr)
\longrightarrow H^0(X_0,mK_{X_0})
\]
is surjective. Equivalently, for each positive integer $m$, the $m$-genus $P_m(X_t)$ of $X_t$ is constant, by Theorem \ref{intro-main-criterion}.
\end{theorem}

\begin{proof}
We identify
$K_{X/\Delta}$ with $K_X$ throughout the proof, and
restrictions of pluricanonical sections to $X_0$ are understood
via the corresponding adjunction identification.
Note that the $m=1$ case is trivial.

Fix an integer $m\geq 2$ and a section
$0 \ne s\in H^0(X_0,mK_{X_0})$.
Choose a disk $\Delta_r\Subset\Delta$ centered at $0$ such that
$X_r:=f^{-1}(\Delta_r)$ is K\"ahler and
Proposition~\ref{prop-total-compactness} applies.
Set $L:=K_{X_r}^{\otimes(m-1)}$ and equip it with
the metric $h^{m-1}$.
Its curvature is semipositive.
Since $s\neq0$, the lower bound~\eqref{eq-limit-barrier} shows
that the restricted weight is not identically $-\infty$ on $X_0$.
Thus the restricted metric is well defined, and
Lemma~\ref{lem-fiber-integrability} gives the required integrability.
Theorem~\ref{thm-ot} therefore yields a section
$\widetilde s\in H^0(X_r,mK_{X_r})$
with $\widetilde s|_{X_0}=s$.
\end{proof}

\section{Local stability of canonical nefness}
\label{sec-nef-openness}

The deformation behavior of nef canonical and adjoint canonical bundles
has recently been studied by Li--Rao--Wang~\cite{LRW25} and
Li--Liu--Rao~\cite{LLR26}. We refer to these two papers for earlier
results and further historical background. In particular,
Li--Liu--Rao proved that canonical nefness is constant on the whole disc
for smooth K\"ahler morphisms~\cite[Theorem 1.2]{LLR26}. Independently,
Wang proved the corresponding local result in fiber dimension four for
smooth weakly K\"ahler morphisms~\cite{Wang26}.

\begin{theorem}\label{thm-nef-openness}
Let $\pi:X\to \Delta$ be a smooth proper surjective
holomorphic map of complex manifolds with connected fibers
$X_t:=\pi^{-1}(t)$.
Suppose that $X_0$ is K\"ahler and that $K_{X_0}$ is nef.
Then there is $\varepsilon>0$ such that $K_{X_t}$ is nef for each
$t\in \Delta_{\varepsilon}:=\{t\in\Delta:|t|<\varepsilon\}$.
There is no restriction on the fiber dimension.
The total space $X$ is not assumed to be K\"ahler.
\end{theorem}
We prove Theorem~\ref{thm-nef-openness} by contradiction. A fiber with
non-nef canonical bundle contains a canonical-negative rational curve and
we choose one of least area with respect to a vertically positive closed
real two-form. Stable-map compactness and the nefness of $K_{X_0}$
confine its deformations to that fiber. Deformation theory in the total
space then gives a lower bound for the dimension of the resulting
compact curve family, while a transported real hard-Lefschetz class and
a two-point incidence space give a strictly smaller upper bound. This
argument requires neither a K\"ahler form on the total space nor the
projective contraction furnished by the transcendental base-point-free
theorem of C. Hacon--L. Xie~\cite{HX26}.

All dimensions below are complex dimensions.
Cohomology is real de Rham cohomology unless coefficients are specified.
On a compact K\"ahler manifold, a nef class is a class in the closure of
the K\"ahler cone.
Intersections with an irreducible curve are computed on its normalization.

\begin{lemma}[{\cite[Lemma~3.19 and Remark~3.20]{LRW25}}]
\label{lem-negative-curves}
Let $M$ be a connected compact K\"ahler manifold.
If $K_M$ is not nef, there is an irreducible rational curve $C\subset M$
such that $K_M\cdot C<0$.
\end{lemma}

\begin{proof}
The case of curves is immediate.
In higher dimension, W. Ou proved that a compact K\"ahler manifold has
pseudoeffective canonical bundle if and only if it is not uniruled~\cite[Theorem 1.1]{Ou25}.
See also the alternative proof by Cao--P\u{a}un~\cite[Corollary 5.3 and Section 6]{CP25}.
Thus the lower-dimensional hypothesis in~\cite[Theorem 1.3]{CH20} holds in all dimensions.
That theorem gives the required curve when $K_M$ is pseudoeffective
but not nef.

Suppose that $K_M$ is not pseudoeffective.
The same criterion of Ou and Cao--P\u{a}un shows that $M$ is uniruled.
Work locally on a smooth parameter space for a covering family of
rational curves, and trivialize the family of their normalizations.
This gives a holomorphic family of maps from $\mathbb P^1$ whose
evaluation dominates an open subset of $M$.
At a suitable smooth parameter and a suitable point of the source,
the evaluation map has surjective differential.
The parameter directions give global sections of $f^*T_M$.
The source direction is also obtained from global vector fields on
$\mathbb P^1$.
Thus, for a suitable nonconstant map $f:\mathbb P^1\to M$,
$H^0(\mathbb P^1,f^*T_M)$ generates $f^*T_M$ at a point.
Write $f^*T_M\cong\bigoplus_{i=1}^{\dim M}\mathcal O_{\mathbb P^1}(a_i)$.
Generation at a point gives $a_i\geq0$ for each $i$.
The nonzero differential
$T_{\mathbb P^1}\cong\mathcal O_{\mathbb P^1}(2)\to f^*T_M$
forces some $a_i\geq2$.
Therefore $\deg f^*K_M=-\sum_i a_i<0$.
The reduced image of $f$ is the required rational curve.
\end{proof}

\begin{lemma}\label{lem-closed-form}
Let $\pi:X\to \Delta$ be a smooth proper holomorphic map,
and suppose that $X_0$ is compact K\"ahler.
Fix a K\"ahler form $\omega_0$ on $X_0$.
After shrinking $\Delta$, all fibers are K\"ahler and there is a smooth
closed real two-form $\Omega$ on $X$ such that
$\Omega|_{X_0}=\omega_0$ and
\begin{equation}\label{eq-vertical-positive}
 \Omega(v,Jv)>0
 \qquad\text{for each nonzero real vertical tangent vector }v.
\end{equation}
Here $J$ is the complex structure on $X$.
The form $\Omega$ need not have type $(1,1)$.
\end{lemma}

\begin{proof}
Kodaira--Spencer's local stability of K\"ahler structures~\cite{KS60} gives a smaller disc on which all fibers are
K\"ahler.
Ehresmann's theorem gives a smooth trivialization
$\Phi:X_0\times \Delta\to X$ over a smaller disc.
We choose it to be the identity on $X_0$; see~\cite[Theorem~9.3 and Remark~9.4]{Voi02}.
Define $\Omega$ by
$\Phi^*\Omega=\operatorname{pr}_1^*\omega_0$.
This form is closed and has the required central restriction.
At $t=0$, inequality~\eqref{eq-vertical-positive} holds.
Compactness of $X_0$ and continuity of the vertical complex structures
show that it still holds after one further shrinking of $\Delta$.
\end{proof}

We refer the readers to \cite{Gr85} \cite{MS12} \cite{Sa99} \cite{Sie99} more details about stable map, Gromov's compact theorem and reduced coarse space of unmarked
genus-zero stable maps.
\begin{definition}
A \emph{stable map} to a complex manifold \(Y\) is a holomorphic map
\[
f : (C, p_1, \dots, p_k) \longrightarrow Y
\]
such that:

\begin{center}
\(C\) is a connected complete nodal curve, \\
\(p_1, \dots, p_k\) are distinct smooth marked points, \\
\(f_*[C] = \beta,\)
\end{center}
and the automorphism group
\[
\operatorname{Aut}(C, p_1, \dots, p_k, f)
\]
is finite.

An \emph{automorphism} here is a biholomorphism
\[
\varphi : C \to C
\]
fixing every marked point and satisfying
\[
f \circ \varphi = f.
\]
\end{definition}
We will use unmarked genus-zero stable maps.
Their domains are connected compact nodal curves of arithmetic genus
zero, so their irreducible components are smooth rational curves and
their dual graphs are trees.
A component on which the map is constant must have at least three nodes.
Maps are identified by isomorphisms of their domains.

\begin{lemma}\label{lem-relative-maps}
Let $p:Y\to \Delta$ be a smooth proper holomorphic map with connected
fibers, and let $\Omega$ be a smooth closed real two-form on $Y$
that satisfies~\eqref{eq-vertical-positive} for $p$.
Fix a smooth trivialization over $\Delta$ and a class
$\beta\in H_2(Y_0,\mathbb Z)$.
Let $\mathcal M_{\beta}$ be the reduced coarse space of unmarked
genus-zero stable maps to fibers of $p$ in the transported class $\beta$.
Then $\mathcal M_{\beta}$ is a complex space, and its parameter map
$q_{\beta}:\mathcal M_{\beta}\to \Delta$ is proper.

Suppose, in addition, that $L$ is a holomorphic line bundle on $Y$,
that $L|_{Y_0}$ is nef, and that
$\langle c_1(L|_{Y_0}),\beta\rangle<0$.
Then $q_{\beta}(\mathcal M_{\beta})$ is a closed discrete subset of $\Delta$.
In particular, let $f:\mathbb P^1\to Y_b$ represent $\beta$.
The reduced local germ of $\operatorname{Hom}(\mathbb P^1,Y)$ at $[f]$
is supported on maps to $Y_b$.
\end{lemma}

\begin{proof}
\smallskip
\noindent\hypertarget{step-relative-maps-space}{\textbf{Step 1. The
relative stable-map space and its parameter map.}}
The analytic stable-map construction applies to any complex target~\cite[Proposition 1.2]{Sie99}.
It gives an analytic orbispace, with a complex space as its coarse space.
Genus, the number of marked points, and the homology class are locally
constant in this construction.
Taking the indicated part and then its reduction gives
$\mathcal M_{\beta}$.
Any holomorphic map from a connected compact nodal curve to the disc
is constant.
Thus a stable map to $Y$ lies in one fiber of $p$.
The value of this constant defines the holomorphic map $q_{\beta}$.
On a local universal family, it can be computed by evaluation at a local
section through a smooth point of the domain.
These local definitions agree and descend to the coarse space.

\smallskip
\noindent\hypertarget{step-relative-maps-area}{\textbf{Step 2. A
uniform area bound over compact subsets of the base.}}
To prove properness of the parameter map constructed in
\stepref{step-relative-maps-space}{1}, fix a compact subset $K\subset \Delta$
and a smooth Hermitian metric on $Y$,
with associated real $(1,1)$-form $\gamma$.
On the compact set $p^{-1}(K)$, vertical positivity gives a constant
$A_K>0$ such that
$\gamma(v,Jv)\leq A_K\Omega(v,Jv)$ for vertical real tangent vectors.
Thus a stable map $u:D\to Y_t$ with $t\in K$ satisfies
\begin{equation*}
 \int_Du^*\gamma
 \leq A_K\int_Du^*\Omega
 =A_K\langle[\Omega|_{Y_0}],\beta\rangle.
\end{equation*}
Integrals on a nodal curve mean sums over its components.
The last equality follows from $d\Omega=0$ and the fixed homology class.

\smallskip
\noindent\hypertarget{step-relative-maps-properness}{\textbf{Step 3.
Gromov compactness and properness.}}
Using the uniform estimate obtained in
\stepref{step-relative-maps-area}{2}, for a sequence in
$q_{\beta}^{-1}(K)$, the images lie in the compact
set $p^{-1}(K)$ and the Hermitian areas are uniformly bounded.
The Gromov compactness theorem for stable maps therefore applies;
see~\cite[Theorem 1.1]{IS00}, with empty boundary.
The target complex structure is fixed, so the convergence condition on
complex structures in that theorem is automatic.
After a subsequence, the maps converge to a stable holomorphic map.
The genus and the total homology class are preserved.
If the base parameters converge to $t$, the limit has image in $Y_t$.
It therefore defines a point of $q_{\beta}^{-1}(K)$.
The analytic stable-map charts have the usual stable-map topology;
see~\cite[Sections 1.1 and 2]{Sie99}.
This proves properness of $q_{\beta}$.
Only compact containment and the area bound were used.
No K\"ahler form on $Y$ is required.

\smallskip
\noindent\hypertarget{step-relative-maps-confinement}{\textbf{Step 4.
Discreteness and confinement of local deformations.}}
Set $d_L:=\langle c_1(L|_{Y_0}),\beta\rangle<0$.
The degree of $L$ on any map in $\mathcal M_{\beta}$ equals $d_L$.
There is no such map over $0$.
Indeed, the pushforward of its domain would be an effective curve
cycle on $Y_0$ with negative $L|_{Y_0}$-degree.
This contradicts nefness.
By the properness proved in
\stepref{step-relative-maps-properness}{3}, Remmert's proper mapping
theorem now shows that
$q_{\beta}(\mathcal M_{\beta})$ is a proper closed analytic subset of
the disc, and thus is discrete.

Finally, a sufficiently small deformation of $f$ has the same class
$\beta$ and remains nonconstant, since its $L$-degree is negative.
It therefore gives a point of $\mathcal M_{\beta}$.
Its base parameter lies in the discrete set just found and is close to
$b$, so it equals $b$.
On a reduced parameter germ, this pointwise statement implies that the
universal map factors holomorphically through $Y_b$.
\end{proof}

\begin{lemma}\label{lem-minimal-area}
Let $M$ be a compact complex manifold, and let $\Omega_M$ be a smooth
closed real two-form such that $\Omega_M(v,Jv)>0$ for $v\ne0$.
Suppose that $M$ contains a $K_M$-negative irreducible rational curve.
Then there is such a curve $C$ of least $\Omega_M$-area.
Put $a:=\int_C\Omega_M$ and $\beta:=[C]\in H_2(M,\mathbb Z)$.
Any unmarked genus-zero stable map to $M$ in class $\beta$ has domain
$\mathbb P^1$ and is birational onto its image.

The corresponding reduced stable-map space $\mathcal H_{\beta}$ is
compact.
It has a universal holomorphic $\mathbb P^1$-bundle
$p_0:U_0\to\mathcal H_{\beta}$ and an evaluation map $e_0:U_0\to M$.
Distinct points of $\mathcal H_{\beta}$ determine distinct reduced
image curves.
\end{lemma}

\begin{proof}
\smallskip
\noindent\hypertarget{step-minimal-area-selection}{\textbf{Step 1.
Existence of a least-area negative rational curve.}}
Choose a negative rational curve of area $A$.
Take smooth closed real two-forms $\theta_1,\ldots,\theta_N$ whose
classes form a basis of $H^2(M,\mathbb Z)$ modulo torsion.
For each $j$, compactness and positivity on complex lines give a
constant $B_j>0$ such that
\begin{equation*}
 \left|\int_D\theta_j\right|
 \leq B_j\int_D\Omega_M
\end{equation*}
for any effective compact curve $D$.
For curves of area at most $A$, the integers $\int_D\theta_j$ are
bounded.
Only finitely many vectors of these integers can occur.
They determine the real homology class of $D$, and thus its
$\Omega_M$-area.
The nonempty set of areas of negative rational curves of area at most
$A$ is therefore finite.
Choose $C$ attaining its minimum.
It also minimizes the area among all negative rational curves.

\smallskip
\noindent\hypertarget{step-minimal-area-decomposition}{\textbf{Step 2.
The nonconstant part of a stable map.}}
For the curve chosen in \stepref{step-minimal-area-selection}{1}, let
$u:D\to M$ be a stable map in class $\beta$.
Write $C_i$ for the reduced image of a nonconstant component, and
$m_i\geq1$ for its degree onto $C_i$.
The index $i$ runs over the nonconstant domain components.
In particular, repeated image curves are counted separately.
Each $C_i$ is rational, and
\begin{equation}\label{eq-stable-decomposition}
 \sum_i m_i\int_{C_i}\Omega_M=a,
 \qquad
 \sum_i m_iK_M\cdot C_i=K_M\cdot C<0.
\end{equation}
Some $C_{i_0}$ has negative canonical degree, so
$\int_{C_{i_0}}\Omega_M\geq a$ by minimality.
All nonconstant components have positive area.
Thus~\eqref{eq-stable-decomposition} forces the component indexed by $i_0$ to be the only
nonconstant domain component, with $m_{i_0}=1$.

\smallskip
\noindent\hypertarget{step-minimal-area-domain}{\textbf{Step 3. The
domain and the map.}}
By \stepref{step-minimal-area-decomposition}{2}, there is exactly one
nonconstant component and its degree onto its image is one.
There are no contracted components.
Otherwise, the dual tree would have at least two vertices and only
one noncontracted vertex.
It would have a contracted leaf, which has only one node.
This contradicts stability.
Thus $D\cong\mathbb P^1$ and $u$ is birational onto its image.
Its automorphism group is trivial.

\smallskip
\noindent\hypertarget{step-minimal-area-family}{\textbf{Step 4.
Compactness and the universal family.}}
Using the description obtained in
\stepref{step-minimal-area-domain}{3}, we apply
Lemma~\ref{lem-relative-maps} to the constant family
$M\times\Delta\to\Delta$ with the pullback of $\Omega_M$.
Its fiber over $0$ gives the compact space $\mathcal H_{\beta}$.
Since all stabilizers are trivial, the universal orbi-family from~\cite[Proposition 1.2]{Sie99} is an ordinary universal analytic family.
All its fibers are smooth copies of $\mathbb P^1$.
The family is locally trivial: after choosing three local disjoint
sections, the fiber coordinate taking them to $0$, $1$, and $\infty$
varies holomorphically.
Thus $p_0$ is a holomorphic $\mathbb P^1$-bundle.
Finally, a degree-one map from $\mathbb P^1$ to its reduced image is
the normalization map.
The reduced image therefore determines the stable map up to isomorphism.
\end{proof}

\begin{lemma}\label{lem-two-points}
Let $Y$ be a complex manifold and $H_0$ a compact reduced complex space.
Let $p_0:U_0\to H_0$ be a holomorphic $\mathbb P^1$-bundle, and let
$e_0:U_0\to Y$ be holomorphic.
Suppose that the fiber maps are birational onto their images and that
distinct parameters give distinct reduced image curves.
Then
\begin{equation*}
 E_0:U_0\times_{H_0}U_0\longrightarrow Y\times Y,
 \qquad
 (u_1,u_2)\longmapsto(e_0(u_1),e_0(u_2))
\end{equation*}
has finite fibers over pairs $(x,y)$ with $x\ne y$.
After restricting to an irreducible component of $H_0$ and pulling back
to a resolution, the two-point evaluation is generically finite onto
its image.
\end{lemma}

\begin{proof}
\smallskip
\noindent\hypertarget{step-two-points-family}{\textbf{Step 1. A
ruled family from a positive-dimensional fiber.}}
Suppose that $E_0^{-1}(x,y)$ has a positive-dimensional irreducible
component for some $x\ne y$.
Let $B$ be a resolution of its reduction.
Then $B$ is a connected compact complex manifold of positive dimension.
The induced map $B\to H_0$ is nonconstant.
Indeed, a fixed nonconstant map from $\mathbb P^1$ has only finitely
many pairs of preimages of $(x,y)$.

Pull back the universal bundle to obtain $p_B:V\to B$ and $e_B:V\to Y$.
The chosen preimages give disjoint sections $s_0,s_{\infty}$, with
$e_B\circ s_0=x$ and $e_B\circ s_{\infty}=y$.
Trivializations taking the sections to $0$ and $\infty$ have linear
transition maps in the fiber coordinate.
Thus $V\cong\mathbb P(\mathcal O_B\oplus L)$ for a holomorphic line
bundle $L$ on $B$.

\smallskip
\noindent\hypertarget{step-two-points-torsion}{\textbf{Step 2. The
line bundle $L$ is torsion.}}
For the ruled family constructed in \stepref{step-two-points-family}{1},
choose a coordinate neighborhood of $x$ in $Y$.
The map $e_B$ sends a neighborhood of the whole section $s_0(B)$
into this coordinate neighborhood.
Expand its coordinate functions in the fiber coordinate along $s_0$.
The coefficients of order $k\geq1$ are global sections of $L^{-k}$,
after fixing the convention for $L$.
This follows from the linear transition maps of the fiber coordinates.
For each $b\in B$, at least one coefficient is nonzero at $b$.
Otherwise, the fiber map would be constant near $s_0(b)$, and thus
constant on that fiber.
Compactness gives finitely many coefficients with no common zero.
Taking powers to a common degree shows that $L^{-m}$ is generated by
global sections for some $m>0$.

Expansion along $s_{\infty}$ gives the same conclusion for $L^m$,
after replacing $m$ by a common multiple.
Choose sections $u\in H^0(B,L^m)$ and $v\in H^0(B,L^{-m})$ that are
both nonzero at a fixed point of $B$.
Their product is a nonzero holomorphic function on the connected compact
manifold $B$, and thus is a nonzero constant.
Both sections are nowhere zero, so $L^m\cong\mathcal O_B$.

\smallskip
\noindent\hypertarget{step-two-points-contradiction}{\textbf{Step 3.
Trivialization and the identity-theorem contradiction.}}
Using the torsion property obtained in
\stepref{step-two-points-torsion}{2}, take the finite \'etale cover
defined by an $m$th root of a trivialization of $L^m$; it trivializes
$L$.
Choose a connected component $B'$ of this cover.
It maps onto $B$, and the pulled-back bundle is
$\mathbb P^1\times B'$.
By compactness, a fixed small disc about $0\in\mathbb P^1$, times
$B'$, maps into the chosen coordinate neighborhood of $x$.
For a fixed fiber coordinate in that disc, the target coordinates are
holomorphic functions on $B'$, so they are constant.
Thus all fiber maps agree on the disc.
They agree on $\mathbb P^1$ by the identity theorem.
Their images are therefore independent of the parameter.
This contradicts the nonconstant map $B\to H_0$ and the hypothesis
on distinct image curves.

\smallskip
\noindent\hypertarget{step-two-points-conclusion}{\textbf{Step 4.
Finiteness and passage to a resolution.}}
The contradiction in \stepref{step-two-points-contradiction}{3} rules
out every positive-dimensional fiber over a pair of distinct points.
The fibers of $E_0$ over $x\ne y$ are consequently compact and
zero-dimensional, and thus finite.
Their locus is nonempty because the fiber maps are nonconstant.
A resolution of the parameter space is an isomorphism over a dense
open subset, so generic finiteness persists after this pullback.
\end{proof}

\begin{proposition}\label{prop-lefschetz-bound}
Let $M$ be a compact K\"ahler manifold of dimension $n$.
Suppose that $\alpha\in H^2(M,\mathbb R)$ satisfies all hard Lefschetz
isomorphisms
\begin{equation}\label{eq-hard-lefschetz}
 \alpha^{n-j}\smile:
 H^j(M,\mathbb R)\longrightarrow H^{2n-j}(M,\mathbb R),
 \qquad 0\leq j\leq n.
\end{equation}
Let $p:P\to B$ be a holomorphic $\mathbb P^1$-bundle over a connected
compact complex manifold, and let $e:P\to M$ be holomorphic.
Assume that the fiber maps are nonconstant, that
$\int_{p^{-1}(b)}e^*\alpha=0$ for each $b\in B$, and that the two-point
evaluation $P\times_B P\to M\times M$ is generically finite onto its image.
Then $\dim B\leq n-2$.
The class $\alpha$ need not have type $(1,1)$, and $B$ need not be
K\"ahler.
\end{proposition}

\begin{proof}
Put $h:=\dim B$, $W:=e(P)$, $d:=\dim W$, and $R:=P\times_B P$.
The spaces $P$ and $R$ are connected compact complex manifolds, and
$W$ is an irreducible analytic subset of $M$.
Write $e_1,e_2:R\to M$ for the two evaluation maps.
Set $r:=h+2-d$.
Since $\dim R=h+2$ and the two-point evaluation is generically finite,
$h+2\leq2d$.
Also $d\leq h+1$, so $1\leq r\leq d$.

The class $\tfrac12c_1(T_{P/B})$ has degree one on each fiber of $p$.
The real Leray--Hirsch theorem therefore applies to $p$.
The degree-zero assumption gives $e^*\alpha=p^*\zeta$ for some
$\zeta\in H^2(B,\mathbb R)$.
Pulling back to $R$ yields
\begin{equation}\label{eq-pullback-equality}
 e_1^*\alpha=e_2^*\alpha.
\end{equation}
This is an equality of real cohomology classes.

Choose a K\"ahler form $\omega_M$ on $M$.
The current
$T:=(e_2)_*(e_1^*\omega_M^r)$ is positive and closed.
It has bidimension $(d,d)$ and is supported on $W$.
The support theorem gives $T=c[W]$ for some $c\geq0$;
see~\cite[Chapter III, Corollary 2.14]{Dem12}.
In fact $c>0$, since
\begin{equation*}
 c\int_W\omega_M^d
 =\int_R e_1^*\omega_M^r\wedge e_2^*\omega_M^d>0.
\end{equation*}
To check strict positivity, work where $(e_1,e_2)$ has rank $h+2$
and $e_2$ has rank $d$.
There $e_1$ is injective on the $r$-dimensional kernel of $de_2$.
The integrand is therefore strictly positive there and is nonnegative
on all of $R$.

Let $k\geq0$ be an integer with $k+r>d$.
The projection formula and~\eqref{eq-pullback-equality} give
\begin{equation}\label{eq-incidence-vanishing}
 \begin{aligned}
 c\,\alpha^k\smile[W]
 &= (e_2)_*\bigl(e_2^*\alpha^k\smile e_1^*[\omega_M]^r\bigr)\\
 &= (e_2)_*e_1^*\bigl(\alpha^k\smile[\omega_M]^r\bigr)=0.
 \end{aligned}
\end{equation}
For the last equality, choose any smooth closed real two-form $A$
representing $\alpha$.
The form $A^k\wedge\omega_M^r$ has real degree $2(k+r)$, whereas
the real rank of $e_1$ is at most $2d$.
Its pullback is therefore zero.
This argument does not require $A$ to have type $(1,1)$.
Nor does it require equality between chosen form representatives in~\eqref{eq-pullback-equality}.

Suppose that $h\geq n-1$.
Then $2d\geq n+1$, and $k:=2d-n$ is a positive integer.
Moreover, $k+r-d=h+2-n>0$.
Equation~\eqref{eq-incidence-vanishing} gives
$\alpha^{2d-n}\smile[W]=0$.
But~\eqref{eq-hard-lefschetz}, with $j=2(n-d)$, says that
\begin{equation*}
 \alpha^{2d-n}\smile:
 H^{2(n-d)}(M,\mathbb R)\longrightarrow H^{2d}(M,\mathbb R)
\end{equation*}
is an isomorphism.
The integer $j$ lies between $0$ and $n$ because $d\leq n$ and
$2d\geq n+1$.
Finally, $[W]\ne0$ since $\int_W\omega_M^d>0$.
This is a contradiction.
\end{proof}

\begin{proof}[Proof of Theorem~\ref{thm-nef-openness}]
\smallskip
\noindent\hypertarget{step-nef-setup}{\textbf{Step 1. Reduction and a
vertically positive closed form.}}
Put $n:=\dim X_0$.
The case $n=0$ is immediate.
For $n=1$, the genus is constant in a smooth proper family, and a
canonical bundle on a smooth compact curve is nef exactly when the
genus is at least one.
We assume $n\geq2$.

Choose a K\"ahler form $\omega_0$ on $X_0$ and apply
Lemma~\ref{lem-closed-form}.
Replace $\Delta$ by the resulting disc and $X$ by its preimage.
Thus all fibers are K\"ahler, and $\Omega$ is closed and positive on
vertical complex lines.
This shrinking is fixed throughout the proof.
It does not depend on a curve class or on a parameter chosen later.
Write $\Omega_t:=\Omega|_{X_t}$.

\smallskip
\noindent\hypertarget{step-nef-lefschetz}{\textbf{Step 2. Flat
hard-Lefschetz classes.}}
Using the form constructed in \stepref{step-nef-setup}{1}, for each real
number $\lambda>0$, we consider the classes
\begin{equation*}
 \alpha_{t,\lambda}
 :=c_1(K_{X_t})+\lambda[\Omega_t].
\end{equation*}
They are the restrictions of
$c_1(K_{X/\Delta})+\lambda[\Omega]$.
Thus they form a flat section of $R^2\pi_*\mathbb R$.
At $t=0$, the class
$\alpha_{0,\lambda}=c_1(K_{X_0})+\lambda[\omega_0]$ is K\"ahler.
Hard Lefschetz on $X_0$, followed by parallel transport of the cohomology
ring, shows that $\alpha_{t,\lambda}$ satisfies~\eqref{eq-hard-lefschetz} on $X_t$ for each $t$.
This statement concerns real cohomology only.

\smallskip
\noindent\hypertarget{step-nef-curve}{\textbf{Step 3. A minimal negative
rational curve.}}
Suppose that $F:=X_b$ has non-nef canonical bundle.
Lemma~\ref{lem-negative-curves} gives a $K_F$-negative rational curve.
By Lemma~\ref{lem-minimal-area}, choose one of least $\Omega_b$-area.
Denote it by $C$, and put
\begin{equation*}
 a:=\int_C\Omega_b>0,
 \qquad
 \ell:=-K_F\cdot C\in\mathbb Z_{>0},
 \qquad
 \lambda:=\frac{\ell}{a}.
\end{equation*}
Set $\alpha:=\alpha_{b,\lambda}$.
Then $\alpha\cdot C=0$, and $\alpha$ satisfies hard Lefschetz by
\stepref{step-nef-lefschetz}{2}.

\smallskip
\noindent\hypertarget{step-nef-confinement}{\textbf{Step 4. Confinement
of deformations to the bad fiber.}}
For the curve selected in \stepref{step-nef-curve}{3}, let
$f:\mathbb P^1\to C\subset F$ be the normalization map.
Transport $[C]$ to a class $\beta\in H_2(X_0,\mathbb Z)$.
Apply Lemma~\ref{lem-relative-maps} with $L=K_{X/\Delta}$.
Its hypotheses hold because the degree of $L$ on this class is
$-\ell<0$, while $L|_{X_0}=K_{X_0}$ is nef.
It follows that the reduced local germ of
$\operatorname{Hom}(\mathbb P^1,X)$ at $[f]$ is supported on
maps to $F$.

\smallskip
\noindent\hypertarget{step-nef-lower-bound}{\textbf{Step 5. The lower
bound for the compact curve family.}}
By \stepref{step-nef-confinement}{4}, the entire reduced local germ is
supported on maps to $F$.
We now estimate the dimension of this germ in the total space.
For a holomorphic map from $\mathbb P^1$ to a complex manifold, the
tangent space is $H^0(\mathbb P^1,f^*T_{X})$ and a complete
obstruction space is $H^1(\mathbb P^1,f^*T_{X})$.
Local lifts over a small extension exist because the target is smooth.
Their differences form a \v{C}ech cocycle in $f^*T_{X}$ tensored
with the kernel of that extension.
Its cohomology class vanishes exactly when the lifts can be glued.
The local Kuranishi model is the zero set of a holomorphic map from
a neighborhood of zero in $H^0(\mathbb P^1,f^*T_{X})$
to $H^1(\mathbb P^1,f^*T_{X})$.
Its codimension is at most the dimension of the latter space.
This gives the estimate below; see~\cite[Section 2]{Sie99} and~\cite[Lemma 5.1]{LLR26}.
All these constructions take place near the compact graph of $f$.
They do not require the target to be compact or K\"ahler.

Since $F$ is a smooth fiber over a disc, its normal bundle in
$X$ is trivial.
The exact sequence
\begin{equation*}
 0\longrightarrow T_F\longrightarrow T_{X}|_F
 \longrightarrow\mathcal O_F\longrightarrow0
\end{equation*}
gives $\deg f^*T_{X}=\ell$.
Riemann--Roch on $\mathbb P^1$ therefore yields
\begin{equation*}
 \begin{aligned}
 \dim_{[f]}\operatorname{Hom}(\mathbb P^1,X)
 &\geq h^0(\mathbb P^1,f^*T_{X})
       -h^1(\mathbb P^1,f^*T_{X})\\
 &=n+1+\ell.
 \end{aligned}
\end{equation*}
This is the dimension of the analytic germ, not merely of its tangent
space.
Passing to the reduction does not change this dimension.

Choose a reduced irreducible local branch $G$ of dimension at least
$n+1+\ell$ through $[f]$.
Its universal map factors through $F$.
After shrinking $G$, its maps have class $[C]$ in $F$.
For example, this follows by a homotopy through short geodesics between
sufficiently close maps into $F$.
They define points of the compact space $\mathcal H_{[C]}$ from
Lemma~\ref{lem-minimal-area}.
Thus forgetting the parametrization gives a holomorphic map
$G\to\mathcal H_{[C]}$.
The image lies in an irreducible component $H_0$ after another local
shrinking.
Indeed, a neighborhood of $[f]$ in the target has only finitely many
irreducible components, whose inverse images cover $G$.

All these maps are birational onto their images.
Two represent the same stable map exactly when they differ by an
automorphism of $\mathbb P^1$.
The fibers of $G\to H_0$ therefore have dimension at most
$\dim\operatorname{PGL}_2(\mathbb C)=3$.
The dimension inequality for a holomorphic map gives
\begin{equation}\label{eq-family-lower}
 \dim H_0\geq\dim G-3\geq n+\ell-2\geq n-1.
\end{equation}
This does not require the image of $G$ to be closed.
The component $H_0$ itself is compact because $\mathcal H_{[C]}$ is
compact.

\smallskip
\noindent\hypertarget{step-nef-contradiction}{\textbf{Step 6. The
hard-Lefschetz upper bound and the contradiction.}}
We apply the upper-bound mechanism to the component constructed in
\stepref{step-nef-lower-bound}{5}.
Take a resolution $H\to H_0$ and pull back the universal family.
This gives a holomorphic $\mathbb P^1$-bundle $p:U\to H$ and an
evaluation map $e:U\to F$.
Here $H$ is connected and compact, and $\dim H=\dim H_0$.
Lemma~\ref{lem-two-points} makes the two-point evaluation generically
finite onto its image.
All fiber maps represent $[C]$, so $e^*\alpha$ has degree zero on
each fiber of $p$.
Proposition~\ref{prop-lefschetz-bound} yields $\dim H\leq n-2$.
Together with~\eqref{eq-family-lower}, this gives
$n+\ell-2\leq n-2$, contrary to $\ell>0$.
Thus no such fiber $F$ exists.
\end{proof}

\begin{remark}
The class $\alpha$ in the proof need not be nef or of type $(1,1)$.
It is not asserted to be an adjoint threshold class.
Only its hard Lefschetz property and its zero degree on the selected
curve family are used.
Positivity in Proposition~\ref{prop-lefschetz-bound} comes from a
K\"ahler form on the single fiber $F$.
The closed real form $\Omega$ is used only to control areas and to
transport cohomology classes.
The proof uses the negative-curve theorem of Cao--A. H\"oring together
with Ou's criterion and the alternative proof of Cao--P\u{a}un~\cite{CH20,Ou25,CP25}.
It does not use a base-point-free theorem, an extremal contraction,
or a separate deformation theorem for pseudoeffectivity.
\end{remark}

\begin{remark}[Comparison with~\cite{LLR26} and the role of total-space
K\"ahlerness]
Only the local deformation--obstruction estimate for rational curves
is shared with Section~5 of~\cite{LLR26}; compare~\cite[Lemma~5.1 and Proposition~5.3]{LLR26}. The compactness and
confinement arguments are different. In~\cite{LLR26}, a K\"ahler form
on the total space controls the areas of relative cycles, and the proof
is organized around a nef threshold class and its contraction.

Here both ingredients are avoided. Lemma~\ref{lem-closed-form}
provides a closed real two-form $\Omega$ that is positive only in the
vertical directions; it need not be of type $(1,1)$ or positive on the
total space. Since every stable map lies in a fiber, this vertical
positivity gives the uniform area bound required for Gromov compactness.
Lemma~\ref{lem-relative-maps} and Remmert's theorem then confine the
negative curves away from the central fiber. The final dimension
contradiction uses a transported real hard-Lefschetz class and a
K\"ahler form only on the individual fiber under consideration. Thus,
unlike~\cite{LLR26}, the present local argument requires neither a
total-space K\"ahler form nor a contraction. In particular, it does
not use the projective contraction obtained from the transcendental
base-point-free theorem of Hacon--Xie~\cite{HX26}.
\end{remark}

\section{Extension in one-parameter degenerations}
\label{sec-nef-big}

Regarding the nonsmooth-morphism part of Matsumura's Problem \ref{prop-matsu}, in this section we establish the corresponding extension results by means of a vanishing theorem. The bigness assumption is imposed on the restriction of the relative canonical bundle to a single irreducible component of the central fiber. We then utilize the relative nefness of the canonical bundle to transfer this bigness to the nearby smooth fibers and thus show that the morphism is Moishezon near the central fiber. Finally, we apply Matsumura's relative vanishing theorem via a projective modification to establish the extension result.

Throughout this section, 
for a K\"ahler morphism $g:V\to \Delta$ from
a connected complex manifold to a disk, a holomorphic line bundle $L$ on $V$ is called \emph{$g$-nef} if, for each irreducible component $V_{t,i}$ of each 
$(V_t){\rm red}$, the class $c_1(L|_{V_{t,i}})$ is nef in the sense analogous to Definition \ref{def-nef} (e.g., \cite[Definition 2.1]{DHP24})

We first give an observation on vanishing and extension that does not rely on the local stability of nefness in Theorem \ref{thm-nef-openness}. This observation serves as a model  case for Theorem \ref{thm-nef-big}, where the argument is extended to one-parameter degenerations and the bigness assumption is imposed  on the restriction of the relative canonical bundle to a single irreducible component of the central fiber.  Note that the nefness and bigness of the canonical line bundle are stable under small deformations (\cite[Proposition 3.16]{Cam91}). Thus one can also apply the Kawamata--Viehweg vanishing theorem directly to the nearby fibers of the central fiber, without using the semicontinuity theorem.  

\begin{observation}\label{thm-smooth-nefbig}
Let $\pi:X\to\Delta$ be a smooth proper morphism from a complex
manifold of dimension $n+1$, where $n\geq1$, and assume that $X_0$
is K\"ahler. Suppose that $K_{X_0}$ is nef and big.
Then, for each fixed integer $m\geq2$, there exists a nonempty (analytic Zariski) open subset $\Delta'\ni 0$ of $\Delta$
such that  
\[
R^q\pi_*K_{X/\Delta}^{\otimes m}=0
\qquad\text{for each }q>0
\]
on $\Delta'$.
Moreover, for each fixed integer $m\geq1$, the restriction map
\[
H^0(X,K_{X/\Delta}^{\otimes m})
\longrightarrow
H^0(X_0,mK_{X_0})
\]
is surjective.
\end{observation}

\begin{proof}

For any fixed $m\geq 2$, the Kawamata--Viehweg vanishing
theorem  gives
\[
H^q\bigl(X_0,
K_{X_0}\otimes K_{X_0}^{\otimes(m-1)}\bigr)
=
H^q(X_0,mK_{X_0})
=
0
\qquad\text{for each }q>0.
\]
Then Grauert's upper semicontinuity theorem and Theorem \ref{intro-main-criterion}
 imply that $R^q \pi_* K_{X / \Delta}^{\otimes m}$ is locally free of rank zero on a nonempty (analytic Zariski) open subset $\Delta'\subseteq \Delta$. Thus  for any $m\geq 2$, we have
$$
R^q \pi_* K_{X / \Delta}^{\otimes m}=0 \quad \text { for each } q>0 
$$
on $\Delta'$.
 In particular, $\pi_* K_{X / \Delta}^{\otimes m}$ is torsion free and thus Theorem \ref{intro-main-criterion} gives rise to 
the surjection 
\[
H^0(X,K_{X/\Delta}^{\otimes m})
\longrightarrow
H^0(X_0,mK_{X_0})
\]
for any $m\geq 2$.    

The surjectivity in the case $m=1$ requires neither the nefness nor the bigness assumption, and follows from the local constancy of the Hodge numbers near  a K\"ahler fiber and Theorem \ref{intro-main-criterion}.  
\end{proof}

For the nonsmooth case of  Matsumura's Problem \ref{prop-matsu}, we replace the fiberwise Kawamata--Viehweg vanishing and semicontinuity argument above by the Boucksom--Demailly--P\u{a}un criterion for bigness, followed by Matsumura's relative vanishing theorem on a suitable projective modification.

\begin{theorem}\label{thm-nef-big}
Let $\pi:X\to \Delta$ be a proper surjective K\"ahler morphism with connected fibers
from a connected complex manifold of dimension $n+1$, where $n\geq1$.
Suppose that the central fiber $X_0:=\pi^{-1}(0)$ is a 
 simple normal crossing divisor and that $K_X$ is
$\pi$-nef. Write $X_0=\sum_{i=1}^aD_i$.
Assume that $K_{X/\Delta}|_{D_i}$ is big for at least one index $i$.
Then the following statements hold.
\begin{enumerate}[(i)]
\item\label{item-nef-big-vanishing}
For each fixed integer $m\geq2$, after shrinking the disk 
centered at $0$, we obtain that  
\begin{equation*}
R^q(\pi)_*K_{X/\Delta}^{\otimes m}=0
\qquad\text{for each }q>0.
\end{equation*}
\item\label{item-nef-big-extension}
For each fixed integer $m\geq1$, after shrinking the disk 
centered at $0$, we obtain that  
the restriction map
\begin{equation*}
H^0(X,K_{X/\Delta}^{\otimes m})
\longrightarrow
H^0(X_0,K_{X/\Delta}^{\otimes m}|_{X_0})
\end{equation*}
is surjective. 
\end{enumerate}
\end{theorem}

\begin{remark}
No bigness assumption is imposed on the other components or on
the intersection strata.
The restriction $K_{X/\Delta}|_{D_i}$ in this theorem is the adjoint bundle
$K_{D_i}\otimes\mathcal O_{D_i}(B_i)$, where
$B_i:=\sum_{j\neq i}(D_j|_{D_i})$ is the reduced boundary divisor.
The bigness or nefness of $K_{D_i}$ does not follow. A simple
counterexample is provided by a stable degeneration of curves. Let
\[
  D_i \simeq \mathbb{P}^1
\]
meet the other components of the central fiber in three nodes. Then
\[
  B_i=p_1+p_2+p_3,
  \qquad
  K_{X/\Delta}|_{D_i}\simeq K_{\mathbb{P}^1}+B_i
  \simeq \mathcal{O}_{\mathbb{P}^1}(1),
\]
which is ample, hence big and nef. Nevertheless,
\[
  K_{D_i}\simeq\mathcal{O}_{\mathbb{P}^1}(-2)
\]
is neither nef nor big. Such a configuration occurs in a semistable
family of stable curves in which the relative canonical bundle \(K_{X/\Delta}\)
is relatively ample.

\end{remark}

The following variation of Chow's Lemma is needed to lift a morphism to a projective morphism.

\begin{lemma}\label{lem-projective-model}
Let $p:V\to \Delta$ be a proper surjective holomorphic map from
a connected complex manifold, and let $B$ be a line bundle on $V$.
Suppose that there is a nonempty open subset $\Delta^\circ\subset \Delta$
over which $p$ is smooth with connected fibers, and that
$B|_{V_t}$ is big for each $t\in \Delta^\circ$.
After shrinking $\Delta$ around $0$, there exist a smooth complex
manifold $Y$ and a proper modification $\mu:Y\to V$ such that
$p\circ\mu:Y\to \Delta$ is projective. The modification $\mu$ is
itself projective.
\end{lemma}

\begin{proof}
For each $t\in\Delta^\circ$, the bigness of $B|_{V_t}$ implies that
$B^{\otimes m_t}|_{V_t}$ is very big for some $m_t \geq 1$.  Consequently,  there is an integer
$m\geq 1$ such that the complete linear system $|B^{\otimes m}|_{V_t}|$
gives  a bimeromorphic map for uncountably many
$t\in\Delta^\circ$.

By Grauert's generic cohomology and base change, there is a proper 
analytic subset $T\subset\Delta$ such that
\[
 (p_*B^{\otimes m})\otimes_{\mathcal O_\Delta}\mathbb C(t)
 \cong
 H^0(V_t,B^{\otimes m}|_{V_t})
\]
for each $t\in\Delta\setminus T$.
Since $\Delta$ is Stein and $p_*B^{\otimes m}$ is coherent,
Cartan's theorem A provides finitely many global sections of
$p_*B^{\otimes m}$ whose germs generate its stalk at $0$. Consider the support of the  cokernel of the resulting morphism
$$
\mathcal{O}_{\Delta}^{\oplus(N+1)} \rightarrow p_* B^{\otimes m}.
$$
Then we obtain that after shrinking $\Delta$, these sections generate
$p_*B^{\otimes m}$ over $\Delta$.  Equivalently, their corresponding
sections
\[
 s_0,\ldots,s_N\in H^0(V,B^{\otimes m})
\]
restrict, for each $t\in\Delta\setminus T$, to a spanning set of $H^0(V_t,B^{\otimes m}|_{V_t}).$
They therefore define a relative meromorphic map
\[
\Phi:V\dashrightarrow\mathbb P^N\times\Delta
\]
whose restriction to $V_t$, for $t\notin T$, is induced by the complete linear system $|B^{\otimes m}|_{V_t}|$. Thus $\Phi|_{V_t}$ is bimeromorphic onto its image for uncountably many $t$.

Let $V'\subset\mathbb P^N\times\Delta$ be the image of $\Phi$.
Since $\Phi$ is bimeromorphic on uncountably many fibers, it is
generically one-to-one, and hence
\[
 \Phi:V\dashrightarrow V'
\]
is bimeromorphic. This is exactly the so-called bimeromorphic embedding developed in \cite{RT1,RT22}. Moreover, $V'\to\Delta$ is projective.   The remaining part of this lemma follows  from a standard Chow's Lemma (e.g., \cite[Corollary 2]{Hi75}  or \cite[Corollary 2.9]{Pt94}) argument (e.g., \cite[Lemma 2.20]{CRT25}). 
\end{proof}

The following relative vanishing theorem of Matsumura for projective morphisms plays an important role in the proof of Theorem~\ref{thm-nef-big}.

\begin{theorem}[{\cite[Theorem 1.7]{Mat22}}]\label{thm-relative-kv-vanishing}
Let \(\pi\colon X \to \Delta\) be a surjective projective morphism from
a complex manifold \(X\) to an analytic space \(\Delta\), and let
\((F,h)\) be a possibly singular Hermitian line bundle on \(X\) with
semipositive curvature.

Then we have
\[
  R^q\pi_*\bigl(K_X \otimes F \otimes \mathcal{I}(h)\bigr)=0
\]
for every
\[
  q >
  f-\max_{\substack{t\in\Delta\\
                    \pi\,\text{is smooth at}\,t}}
  \operatorname{nd}\bigl(F|_{X_t},h|_{X_t}\bigr),
\]
where \(f\) is the dimension of the general fibers and $\operatorname{nd}(\bullet, \bullet
)$ is the numerical Kodaira dimension (\cite[Definition 4.1]{Mat22}, \cite[\S 3, 4]{Ca14}). In particular, if
\(\bigl(F|_{X_t},h|_{X_t}\bigr)\) is big for some point \(t\) in the
smooth locus of \(\pi\), then
\[
  R^q\pi_*\bigl(K_X \otimes F \otimes \mathcal{I}(h)\bigr)=0
  \qquad\text{for every } q>0.
\]
\end{theorem}

\begin{lemma}\label{lem-adjoint-vanishing}
Let $f:Y\to \Delta$ be a proper surjective projective
holomorphic map from a connected complex manifold, and let $F$
be an $f$-nef line bundle. Suppose that there is a nonempty
open subset $\Delta^\circ\subset \Delta$ over which $f$ is smooth with
connected fibers and that $F|_{Y_t}$ is big for each $t\in \Delta^\circ$.
After shrinking $\Delta$ around $0$, one has
\[
R^qf_*(K_{Y/\Delta}\otimes F)=0
\qquad\text{for each }q>0.
\]
\end{lemma}

\begin{proof}
Fix an $f$-ample line bundle $A$.
We first produce an effective divisor $E$ and an integer $r>0$
such that
\begin{equation}\label{eq-big-decomposition}
F^{\otimes r}\cong A\otimes\mathcal O_Y(E).
\end{equation}
As in the proof of Lemma~\ref{lem-projective-model}, choose
$s\in\Delta^\circ$   at which cohomology and base change holds for
every positive tensor power. Since $F|_{Y_{s}}$ is big, choose $r$
sufficiently large that
\[
H^0\bigl(Y_{s},
(F^{\otimes r}\otimes A^{-1})|_{Y_{s}}\bigr)\neq0.
\]
Cartan's Theorem~A  together with base
change at $s$, then gives a nonzero section of
\[
H^0\bigl(\Delta,f_*(F^{\otimes r}\otimes A^{-1})\bigr)=H^0(Y,F^{\otimes r}\otimes A^{-1}).
\]
Its zero divisor $E$ induces the decomposition \eqref{eq-big-decomposition}.

After shrinking $\Delta$, choose a smooth Hermitian metric $h_A$ on
$A$ whose normalized curvature $\omega_A$ is positive on $Y$.
We normalize the curvature so that the divisor metric induced by the canonical section of $E$ on
$\mathcal O_Y(E)$ has curvature current $[E]$.
Equation~\eqref{eq-big-decomposition} then gives a singular divisor
metric $h_{\rm big}$ on $F$ such that
\[
\Theta(F,h_{\rm big})=\frac1r\bigl(\omega_A+[E]\bigr).
\]

After shrinking $\Delta$ if necessary, we may choose a fixed number $0<\delta<1$ sufficiently small such that   the divisor metric with coefficient $\delta/r$ has trivial multiplier ideal,  as can be checked on a log resolution of $(Y,E)$.
Now choose an integer $k$ satisfying
$k>r(1-\delta)/\delta$.
Since $F$ is $f$-nef and $A$ is $f$-ample,
$F^{\otimes k}\otimes A$ is $f$-ample.
After a further shrinking, it has a smooth Hermitian metric
$g_k$ with positive curvature on the total space.
The smooth metric
$h_k:=(g_k\otimes h_A^{-1})^{1/k}$ on $F$ satisfies
$\Theta(F,h_k)\geq-\omega_A/k$.
Consequently, the metric
$h:=h_k^{1-\delta}h_{\rm big}^{\delta}$ satisfies
\begin{equation*}
\Theta(F,h)\geq
\left(\frac{\delta}{r}-\frac{1-\delta}{k}\right)\omega_A
+\frac{\delta}{r}[E],
\qquad
\mathcal I(h)=\mathcal O_Y.
\end{equation*}
The coefficient of $\omega_A$ is strictly positive.

Choose a smooth fiber not contained in $E$.
The restriction of $h$ to this fiber is well defined and its
curvature dominates a K\"ahler form. Thus the singular
Hermitian line bundle on that fiber is big, with numerical
dimension equal to the fiber dimension.
Theorem~\ref{thm-relative-kv-vanishing} together with
$\mathcal I(h)=\mathcal O_Y$, gives that
$$R^qf_*(K_Y\otimes F)=0$$ for each $q>0$.
Tensoring with the trivial line bundle $K_\Delta^{-1}$ and using
the projection formula gives the assertion for $K_{Y/\Delta}$.  This completes the proof of Lemma \ref{lem-adjoint-vanishing}. 
\end{proof}

\begin{proof}[Proof of Theorem~\ref{thm-nef-big}]
\smallskip
\noindent\textbf{Step 1. Bigness of the canonical line bundle of the nearby fibers.}

After shrinking $\Delta$, the total space $X$ is K\"ahler.
Fix a K\"ahler form $\Omega$ on $X$ and a smooth closed form
$\alpha$ representing $c_1(K_{X/\Delta})$.
Note that a proper morphism  over a smooth curve is automatically flat. Then $\{X_t\}_{t\in \Delta}$  is an analytic family of $n$-cycles (via the Douady--Barlet morphism).  Consequently, the connectedness of $\Delta$ implies that $X_{t_1}$ is homologous to $X_{t_2}$ for any $t_1,t_2\in \Delta$.
Thus we obtain
\begin{equation*}
\int_{X_t}\alpha^k\wedge\Omega^{n-k}
=\sum_{i=1}^a
\int_{D_i}c_1(K_{X/\Delta}|_{D_i})^k\wedge(\Omega|_{D_i})^{n-k}.
\end{equation*}

Each restriction $K_{X/\Delta}|_{D_i}$ is nef, so each summand is
nonnegative. Recall that for a nef line bundle on a compact K\"ahler manifold, bigness
is equivalent to positive top self-intersection (\cite[Theorem~0.5]{DP04},  \cite[Theorem~4.7]{Bou02}).
Thus $K_{X/\Delta}|_{X_t}\cong K_{X_t}$ is nef with positive top
self-intersection, and is big for each $t\neq 0$ (after shrinking $\Delta$ such that $X_t$ is smooth for any $t\neq 0$).

\smallskip
\noindent\textbf{Step 2. A projective modification and adjoint vanishing.}
Apply Lemma~\ref{lem-projective-model} to $B=K_{X/\Delta}$.
We obtain a smooth $Y$ and a projective proper modification
$\mu:Y\to X$ such that $f:=\pi\circ\mu:Y\to \Delta$ is projective.
Fix $m\geq2$. 
Pullback preserves relative nefness, so $\mu^*K_{X/\Delta}^{\otimes(m-1)}$ is $f$-nef.
On a general smooth fiber, $\mu$ restricts to a proper
modification, and  thus the pullback of the big bundle
$K_{X/\Delta}^{\otimes(m-1)}|_{X_t}$ is big for general $t$. Furthermore, 
the fibers of $f$ are connected, because $\mu_*\mathcal O_Y
\cong\mathcal O_X$ and $\pi$ has connected fibers. Then 
Lemma~\ref{lem-adjoint-vanishing} gives that
\[
R^qf_*(K_{Y/\Delta}\otimes
\mu^*K_{X/\Delta}^{\otimes(m-1)})=0
\qquad\text{for each }q>0,
\]
after shrinking the base again.

\smallskip
\noindent\textbf{Step 3. Descent of the vanishing and extension of sections.}
Since $X$ is smooth, we have $\mu_*K_Y\cong K_X$. Then 
Grauert--Riemenschneider vanishing gives
$R^q\mu_*K_Y=0$ for $q>0$.
The projection formula therefore gives
\begin{equation}\label{eq-gr-descent}
R^q\mu_*(K_{Y/\Delta}\otimes
\mu^*K_{X/\Delta}^{\otimes(m-1)})
\cong
\begin{cases}
K_{X/\Delta}^{\otimes m},&q=0,\\
0,&q>0.
\end{cases}
\end{equation}
The Leray spectral sequence for $f=\pi\circ\mu$ and~\eqref{eq-gr-descent} identifies
$R^q\pi_*K_{X/\Delta}^{\otimes m}$ with
$R^qf_*(K_{Y/\Delta}\otimes
\mu^*K_{X/\Delta}^{\otimes(m-1)})$.
This proves \textup{(\ref{item-nef-big-vanishing})}. As in Observation \ref{thm-smooth-nefbig}, the vanishing yields the extension statement (for the $m=1$ case, \cite[Theorem 6.5]{Tk95} 
 or  \cite[Corollary 1.5]{Mat22} gives the required torsion-freeness of
$R^1\pi_*K_{X/\Delta}$), and this proves \textup{(\ref{item-nef-big-extension})}.
\end{proof}

\section*{Acknowledgements}
{The authors thank Professor I-Hsun Tsai for valuable discussions over the years on extension topics and the geometry of Moishezon morphisms, which led us to extension questions related to Matsumura's problem. The authors also wish to thank Professors Mu-Lin Li and Xiao-Lei Liu for useful discussions on deformation stability of canonical nefness. The authors also wish to thank Professor Mihai P\u{a}un for his interest in this work.  Furthermore, the authors would like to thank Jiawei Feng and Yi Li for helpful discussions on constancy of cohomological dimension.

The authors used ChatGPT 5.6 and 6 Pro as auxiliary tools in addressing the smooth case of Matsumura's problem. Following extensive discussions, ChatGPT 6 Pro suggested using P\u{a}un's twisted form of Schumacher's curvature formula to address the pseudoeffectivity issue, which enabled us to successfully construct the desired singular metric.  DeepSeek was also used to polish the language. The authors take full responsibility for the correctness of this paper.

\end{document}